\documentclass[11pt,a4paper,reqno]{amsart}
\usepackage{amsmath}
\usepackage{amsfonts}
\usepackage{amssymb}
\usepackage{amscd}
\usepackage{url}
\usepackage{enumerate}
\usepackage[pdftex,bookmarks=true]{hyperref}
\usepackage{graphicx}

\newcommand\R{{\mathbf{R}}}

\renewcommand\P{{\mathbf{P}}}
\newcommand\E{{\mathbf{E}}}

\newcommand\tr{{\operatorname{tr}}}

\newcommand\dist{{\operatorname{dist}}}
\newcommand\Z{{\mathbf{Z}}}

\newcommand\col{{\mathbf{c}}}
\newcommand\row{{\mathbf{r}}}

\newcommand\ep{\varepsilon}
\newcommand\eps{\varepsilon}

\newcommand\Ba{{\mathbf a}}
\newcommand\Bb{{\mathbf b}}
\newcommand\Bc{{\mathbf c}}

\newcommand\Be{{\mathbf e}}

\newcommand\Bu{{\mathbf u}}
\newcommand\Bv{{\mathbf v}}
\newcommand\Bw{{\mathbf w}}
\newcommand\Bx{{\mathbf x}}
\newcommand\By{{\mathbf y}}
\newcommand\Bz{{\mathbf z}}
\newcommand\CE{{\mathcal E}}
\newcommand\CF{{\mathcal F}}

\newcommand\CL{{\mathcal L}}

\newcommand\CN{{\mathcal N}}

\newcommand\LCD{\mathbf{LCD}}

\newcommand\Comp{\mathbf{Comp}}
\newcommand\Incomp{\mathbf{Incomp}}
\newcommand\supp{\mathbf{supp}}
\newcommand\spread{\mathbf{spread}}

\newcommand\LCDhat{\widehat{\mathbf{LCD}}}

\theoremstyle{plain}
  \newtheorem{theorem}[subsection]{Theorem}

  \newtheorem{fact}[subsection]{Fact}
  \newtheorem{lemma}[subsection]{Lemma}
  \newtheorem{corollary}[subsection]{Corollary}

  \newtheorem{claim}[subsection]{Claim}

\theoremstyle{definition}
  \newtheorem{definition}[subsection]{Definition}

\begin{document}

\title[Distances in Wigner matrices]{Concentration of distances in Wigner matrices}

\author{Hoi H. Nguyen}
\address{Department of Mathematics, The Ohio State University, Columbus OH 43210}

\email{nguyen.1261@math.osu.edu}
\thanks{The author is supported by NSF grants DMS-1358648, DMS-1128155 and CCF-1412958. Any opinions, findings and conclusions or recommendations expressed in this material are those of the author and do not necessarily reflect the views of  NSF}

\subjclass[2010]{15A52, 15A63, 11B25}

\maketitle

\begin{abstract}  It is well-known that distances in random iid matrices are highly concentrated around their mean. In this note we extend this concentration phenomenon to Wigner matrices. Exponential bounds for the lower tail are also included. 
\end{abstract}

\section{Introduction}
Let $\xi$ be a real random variable of mean zero, variance one, and there exists a parameter $K_0>0$ such that for all $t$

$$\P(|\xi|\ge t)= O(\exp(-t^2/K_0)).$$

Let $A=(a_{ij})$ be a random matrix of size $N$ by $N$ where $a_{ij}$ are iid copies of $\xi$. For convenience, we denote by $\row_i(A)=(a_{i1},\dots, a_{in})$ the $i$-th row vector of $A$.

For a given $1\le n\le N-1$ let $B$ be the submatrix of $A$ formed by $\row_2(A),\dots,\row_{n+1}(A)$ and let $H\subset \R^N$ be the subspace generated by these row vectors. The following concentration result is well known (see for instance \cite[Lemma 43]{TVuniversality}, \cite[Corollary 2.1.19]{Tao-RMT} or \cite[Corollary 3.1]{RV-HW}).

\begin{theorem}\label{thm:distance:independent} With $m=N-n$, we have
$$\P\left(|\dist(\row_1,H) - \sqrt{m}| \ge t \right) \le  \exp(-t^2/K_0^4).$$
\end{theorem}

One can justify Theorem \ref{thm:distance:independent} by applying concentration results of Talagrand or of Hanson-Wright. But all of these methods heavily rely on the fact that $\row_1$ is independent from $H$. In fact, Theorem \ref{thm:distance:independent} holds as long as $H$ is any deterministic non-degenerate subspace.

As Theorem \ref{thm:distance:independent} has found many applications (see for instance \cite{NgV-CLT, TVcir, TV-least, TVuniversality}) and as concentration is useful in Probability in general, it is natural to ask if Theorem \ref{thm:distance:independent} (or its variant) continues to hold when $\row_1$ and $H$ are correlated. We will address this issue for one of the simplest models, the symmetric Wigner ensembles. In our matrix model $A=(a_{ij})$, the upper diagonal entries $a_{ij}, i\le j$ are iid copies of $\xi$. 

\begin{theorem}[Main result, concentration of distance]\label{thm:distance:dependent} With the assumption as above, there exists a sufficiently small positive constant $\kappa$ depending on the subgaussian parameter $K_0$ of $\xi$ such that the following holds for any $m\ge \log^{\kappa^{-1}} N$
$$\P\left(|\dist(\row_1,H) - \sqrt{m}| \ge t  \right)= O\Big(N^{-\omega(1)} + N\exp(- \kappa t)\Big).$$
\end{theorem}

Note that our bound is weaker than Theorem \ref{thm:distance:independent} mainly due to the use of other spectral concentration results (Lemma \ref{lemma:manyvalues'} and Theorem \ref{theorem:concentration}). Theorem \ref{thm:distance:dependent} also implies identical control for the distance from $\row_1$ to other $n$ rows (not necessarily consecutive). We will be also giving tail bounds throughout the proof (see for instance Theorem \ref{theorem:distance:problem} and \ref{theorem:distance:problem'}). 

As we have mentioned, the main bottleneck of our problem is  that $\Bx$ and $H$ are not independent. Roughly speaking, one might guess that the distance of $\row_1=(a_{11},\dots, a_{1n})$ to $H$ is close to the distance from the truncated vector $(a_{12},\dots,a_{1n})$  to the subspace in $\R^{N-1}$ generated by the corresponding truncated row vectors of $B$. The main problem, however, is that the "error term" of this approximation involves a few non-standard statistics of the matrix $P$ which is obtained from $B$ by removing its first column. Notably, we have to resolve the following two obstacles

\begin{enumerate}
\item the operator norm $\|(PP^T)^{-1}\|_2$ must be under control;
\vskip .1in
 \item the sum $\sum_{1\le i\le n}((PP^T)^{-1}P)_{ii}$ must be small; 
\end{enumerate}

For bounding $\|(PP^T)^{-1}\|_2$ in (1), as it is the reciprocal of the least singular value $\sigma_n(P )$ of $P$, we need to find an efficient lower bound for $\sigma_n(P )$. 

When $n=N-1$, $P$ can be viewed as a random square symmetric matrix of size $N-1$ (and so let's pass to $A$). It is known via the work of  \cite{Ng-sym} and \cite{V} that in this case $A$ is non-singular with very high probability. More quantitatively, the result of Vershynin in \cite{V} shows  the following.

\begin{theorem}\label{thm:singularity} There exists a positive constant $\kappa'$ such that the following holds. Let $\sigma_N(A)$ be the smallest singular value of $A$, then for any $\eps<1$,
$$\P(\sigma_N(A) \le \kappa' \eps N^{-1/2}) =O(\eps^{1/8} + e^{-N^{\kappa'}}).$$
\end{theorem}

It is conjectured that the bound can be replaced by $O(\ep + e^{-\Theta(n)})$, but in any case these bounds show that we cannot hope for a good control on $\sigma_N^{-1}$. The situation becomes better when we truncate $A$ as the matrix becomes less singular. In fact this has been observed for quite a long time (see for instance \cite{GN}). Here we will show the following variant of a recent result by Rudelson and Vershynin from \cite{RV-rec}.

\begin{theorem}\label{thm:least}
There exist positive  constants $\delta,\kappa'$ such that the following holds  for any $\eps<1$

$$\P(\sigma_n(B) \le \kappa'\eps m N^{-1/2}) \le \eps^{\delta m} + e^{-N^{\kappa'}}.$$
\end{theorem}

Thus, if $m$ grows to infinity with $N$, this theorem implies that $\|(PP^T)^{-1}\|_2$ is well under control with very high probability. We will discuss a detailed treatment for Theorem \ref{thm:least} starting from Section \ref{section:least:introduction} by modifying the approach by Rudelson-Vershynin from \cite{RV-rec} and by Vershynin from \cite{V}. 

For  the task of controlling $T=\sum_{1\le i\le n}((PP^T)^{-1}P)_{ii}$ in (2), which is the main contribution of our note, there are two main steps. In the first step we show that with high probability $((PP^T)^{-1}P)_{ii}$ is close to $- \frac{\sum_{1\le j\le n-1}((R_iR_i^T)^{-1}R_i)_j}{D_i}$, where $D_i$ has order $N$ and $R_i$ is the matrix obtained from $P$ by removing its $i$-th row and column. On the other hand, in the second step we show that $T$ is close to $\sum_{1\le j\le n-1}((R_iR_i^T)^{-1}R_i)_j$ for any $1\le i \le n$ with high probability. These two results will then imply that $T$ is close to zero. Our detailed analysis will be presented in  Section \ref{section:distance:dependent:1} and  Section \ref{section:distance:dependent:2}.

In short, our proof of Theorem \ref{thm:distance:dependent} uses both spectral concentration and anti-concentration coupled with linear algebra identities. Given the simplicity of the statement and of its non-Hermitian counterpart, perhaps it is natural to seek for more direct proof. 

To complete our discussion, we give here an application of Theorem \ref{thm:distance:dependent} on the delocalization of the normal vectors in random symmetric matrices. 

\begin{corollary}\label{cor:normal} Let $\Bx$ be the normal vector of the subspace generated by $\row_2(A),\dots,\row_N(A)$. Then with overwhelming probability 

$$\|\Bx\|_\infty =O( \frac{\log^{3/2}N}{\sqrt{N}}).$$

\end{corollary}

This is an analog of a result from \cite{NgV-linear} where $A$ is a non-symmetric iid matrix, see also \cite{RV-del}. However, the method in these papers do not seem to extend to our symmetric model. 

Finally, as the $i$-th row vector of the inverse matrix $A^{-1}$ is orthogonal to all other row vectors of index different from $i$ of the matrix $A$, by comparing in each row and then each column, we deduce the following bound on the entries $(A^{-1})_{ij}$ of the inverse marix $A^{-1}$.

\begin{corollary}
With overwhelming probability we have

$$\sup_{1\le i,j\le N} \frac{|(A^{-1})_{ij}|}{\|A^{-1}\|_{HS}} \le \frac{\log^3 N}{N}.$$
\end{corollary}

For terminology, we will use both row and column vectors frequently in this note, here $\row_i(A)$ and $\col_j(A)$ stand for the $i$-th row and $j$-th column of $A$ respectively. Throughout this paper, we regard $N$ as an asymptotic parameter going to infinity. We write $X = O(Y ), X\ll Y$, or $Y\gg X$ to denote the claim that $|X| \le CY$ for some fixed $C$; this fixed quantity $C$ is allowed to depend on other fixed quantities such as the sub-gaussian parameter $K_0$ of $\xi$ unless explicitly declared otherwise. We also write $X \asymp Y$ or $X=\Theta(Y)$ for $X,Y>0$ to denote the claim that $X\ll Y \ll X$. Lastly, to simplify our presentation, a same constant may used in different places even when they can be totally different.


\begin{center} 

\begin{figure}[!ht]
   \centerline{
    \includegraphics[width=0.6\textwidth]{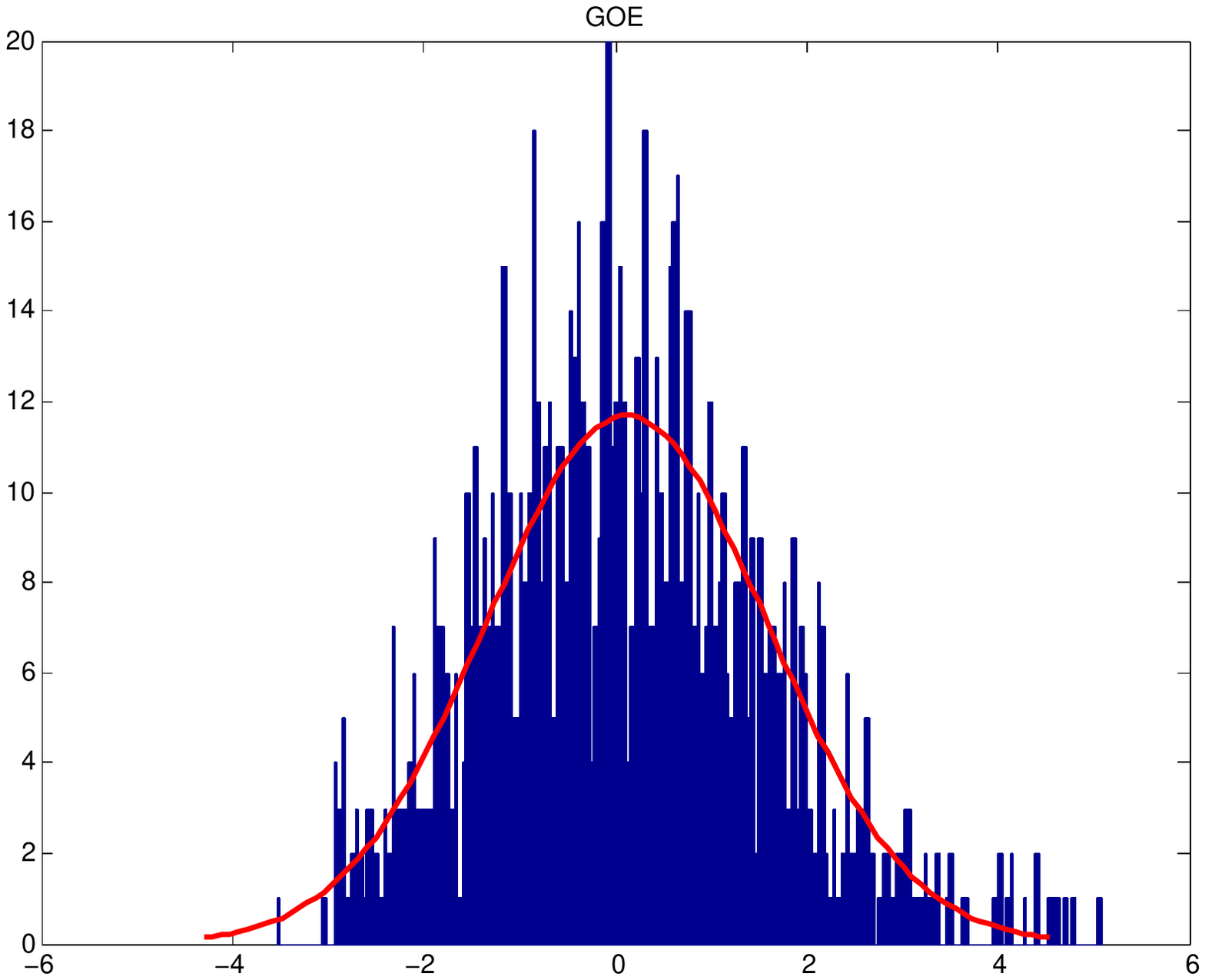} \includegraphics[width=0.6\textwidth]{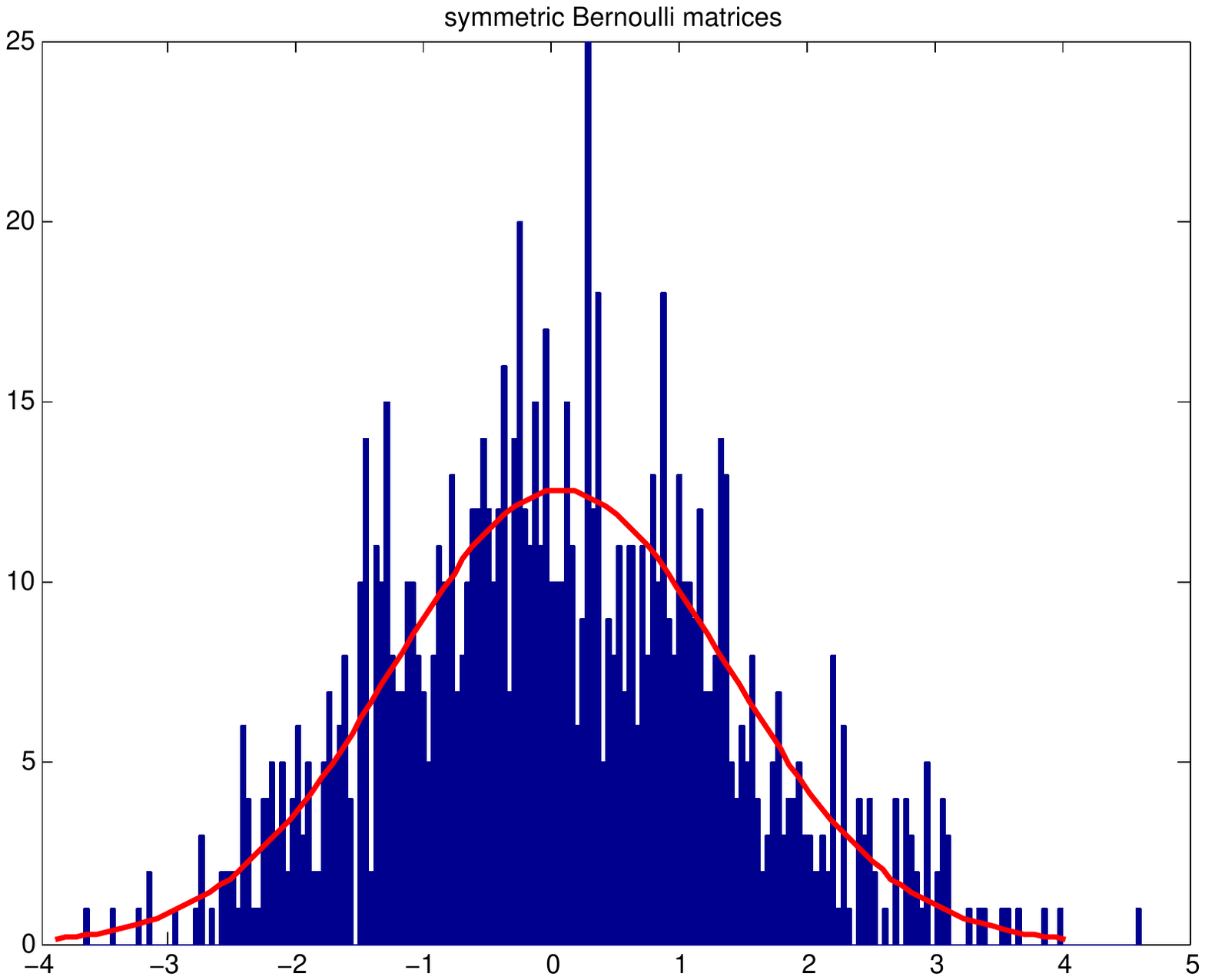}
    } \caption{We sampled 1000 random matrices of size 1000 from GOE and symmetric Bernoulli ensembles.  The histogram represents $(\dist^2-m)/\sqrt{m}$, where the distance is measured from the first row to the subspace generated by the next 900 rows.}
    \label{fig:1}
\end{figure}

\end{center}



\section{Proof of Theorem \ref{thm:distance:dependent}: main ingredients}\label{section:distance:dependent:1}

Let $\Bx=\row_1(A)=(a_{11},\dots,a_{1N})$ be the first row vector of $A$, and $\By_{i}:=\row_{i+1}(A)=(a_{(i+1)1},\dots, a_{(i+1)N})$ be the $(i+1)$-st row vector of $A$ for $1\le i\le n$. Recall that $B$ is the matrix of size $n\times N$ with rows $\By_i, 1\le i\le n$. We can represent $\dist(\row_1,H)=\dist(\Bx,H)$ by the following formula

\begin{equation}\label{eqn:dist:1}
\dist^2(\Bx,H) = \Bx (I_N - B^T (BB^T)^{-1}B) \Bx^T.
\end{equation}

\begin{lemma}\label{lemma:distance:removal}
We have 
$$\dist^2(\Bx,H) = {\Bx^{[1]}} (I_{N-1} -P^T (P P^T)^{-1} P) {\Bx^{[1]}}^T + \frac{[a_{11} - \Bz^T (PP^T)^{-1}P {\Bx^{[1]}}^T]^2}{1+ \Bz^T (PP^T)^{-1} \Bz},$$
where $P$ be the matrix obtained from $B$ by removing its first column $\Bz:=(a_{21},\dots,a_{(n+1)1})^T$, and $\Bx^{[1]}$ is the truncated row vector of $\Bx$, $\Bx^{[1]}=(a_{12},\dots,a_{1N})$, $B= \begin{pmatrix} \Bz & P \end{pmatrix}.$
\end{lemma}

\begin{proof}  A direct calculation shows

\begin{equation}\label{eqn:Q}
Q:=(BB^T)^{-1} = ( P P^T + \Bz \Bz^T)^{-1} = (P P^T)^{-1} - \frac{(PP^T)^{-1} \Bz \Bz^T (P P^T)^{-1}}{1+ \Bz^T (PP^T)^{-1} \Bz}.
\end{equation}

Thus $B^T (BB^T)^{-1}B = B^T Q B$, which has the form 

\[ B^T Q B  =
\left(
\begin{array}{cc}
\Bz^T Q \Bz & \Bz^T Q P \\
P^T Q \Bz    & P^T Q P
\end{array}
\right).
\]

Hence,

$$\Bx (I_N - B^T (BB^T)^{-1} B) \Bx^T = \Bx \Bx^T - a_{11}^2 \Bz^T Q \Bz - 2 a_{11} \Bz^T Q P {\Bx^{[1]}}^T - {\Bx^{[1]}} P^T Q P {\Bx^{[1]}}^T.$$

Using \eqref{eqn:Q}, we obtain the following

$$\Bz^T Q \Bz = \Bz^T (PP^T)^{-1} \Bz - \frac{ (\Bz^T (PP^T)^{-1} \Bz)^2}{1 + \Bz^T (PP^T)^{-1} \Bz} = \frac{\Bz^T (PP^T)^{-1} \Bz}{1 + \Bz^T (PP^T)^{-1} \Bz};$$

as well as 

\begin{align*}
\Bz^T Q P {\Bx^{[1]}}^T &=  \Bz^T [(P P^T)^{-1} - \frac{(PP^T)^{-1} \Bz \Bz^T (P P^T)^{-1}}{1+ \Bz^T (PP^T)^{-1} \Bz}] P{\Bx^{[1]}}^T \\
&=\Bz^T (P P^T)^{-1}P{\Bx^{[1]}}^T - \frac{\Bz^T (PP^T)^{-1} \Bz \Bz^T (P P^T)^{-1} P{\Bx^{[1]}}^T }{1+ \Bz^T (PP^T)^{-1} \Bz}\\
&= \frac{\Bz^T (PP^T)^{-1}P{\Bx^{[1]}}^T }{1+ \Bz^T (PP^T)^{-1}\Bz};
\end{align*}

and

\begin{align*}
 {\Bx^{[1]}} P^T Q P {\Bx^{[1]}}^T &=  {\Bx^{[1]}} P^T [ (P P^T)^{-1} - \frac{(PP^T)^{-1} \Bz \Bz^T (P P^T)^{-1}}{1+ \Bz^T (PP^T)^{-1} \Bz}] P {\Bx^{[1]}}^T\\
 &= {\Bx^{[1]}} P^T (P P^T)^{-1} P {\Bx^{[1]}}^T - \frac{ {\Bx^{[1]}} P^T  (PP^T)^{-1} \Bz \Bz^T (P P^T)^{-1}  P {\Bx^{[1]}}^T}{1+ \Bz^T (PP^T)^{-1} \Bz}. 
\end{align*}

Putting together, one obtains the following 

\begin{align}\label{eqn:dist:formula:general}
\dist^2(\Bx,H)  &= \Bx (I_N - B^T (BB^T)^{-1}B) \Bx^T \nonumber \\ 
 &=  {\Bx^{[1]}} (I_{N-1} -P^T (P P^T)^{-1} P) {\Bx^{[1]}}^T  + a_{11}^2 - a_{11}^2 \frac{\Bz^T (PP^T)^{-1} \Bz}{1 + \Bz^T (PP^T)^{-1} \Bz}\nonumber \\
 &- 2a_{11} \frac{\Bz^T (PP^T)^{-1}P {\Bx^{[1]}}^T }{1+ \Bz^T (PP^T)^{-1}\Bz} +\frac{ {\Bx^{[1]}} P^T  (PP^T)^{-1} \Bz \Bz^T (P P^T)^{-1}  P {\Bx^{[1]}}^T}{1+ \Bz^T (PP^T)^{-1} \Bz} \nonumber \\
 &=  {\Bx^{[1]}} (I_{N-1} -P^T (P P^T)^{-1} P) {\Bx^{[1]}}^T \nonumber \\
 &+ \frac{a_{11}^2 - 2a_{11} \Bz^T (PP^T)^{-1}P {\Bx^{[1]}}^T + {\Bx^{[1]}} P^T  (PP^T)^{-1} \Bz \Bz^T (P P^T)^{-1}  P {\Bx^{[1]}}^T }{1+ \Bz^T (PP^T)^{-1} \Bz} \nonumber \\
& = \Bx^{[1]} (I_{N-1} -P^T (P P^T)^{-1} P) {\Bx^{[1]}}^T + \frac{[a_{11} - \Bz^T (PP^T)^{-1}P {\Bx^{[1]}}^T]^2}{1+ \Bz^T (PP^T)^{-1} \Bz}  \nonumber ,
 \end{align}

proving the lemma.
\end{proof}

Note that the term ${\Bx^{[1]}} (I_{N-1} -P^T (P P^T)^{-1} P) {\Bx^{[1]}}^T$ in Lemma \ref{lemma:distance:removal} is just the squared distance from $\Bx^{[1]}$ to the subspace generated by the rows of $P$ in $\R^{N-1}$. The key difference here is that $\Bx^{[1]}$ is now independent of $P$, so we can apply Theorem \ref{thm:distance:independent}, noting that by Theorem \ref{thm:singularity} with probability at least $1-\exp(-N^{\kappa'})$ the matrix obtained from $A$ by removing its first row and column has full rank, and so the co-dimension of the subspace generated by the rows of $P$ in $\R^{N-1}$ is $N-1-n=m-1$ with very high probability.

\begin{theorem}\label{thm:distance:independent'} Assume that the entries of $A$ have subgaussian parameter $K_0>0$. Then
$$\P\left(|\sqrt{{\Bx^{[1]}}^T (I_{N-1} -P^T (P P^T)^{-1} P) \Bx^{[1]}} - \sqrt{m-1}|\ge \lambda\right) \le \exp(-\lambda^2/K_0^4) + \exp(-N^{\kappa'}).$$
\end{theorem}

This confirms the lower tail bound in Theorem \ref{thm:distance:dependent}. Thus, in order to prove our main result, we must study the remaining term $\frac{[a_{11} - \Bz^T (PP^T)^{-1}P {\Bx^{[1]}}^T]^2}{1+ \Bz^T (PP^T)^{-1} \Bz}$ in Lemma \ref{lemma:distance:removal}.

\begin{theorem}[Bound on the error term, key lemma]\label{thm:distance:error} 
We have
$$\P( \frac{[a_{11} - \Bz^T (PP^T)^{-1}P {\Bx^{[1]}}^T]^2}{1+ \Bz^T (PP^T)^{-1} \Bz} \ge t) = O\Big(N^{-\omega(1)}+N \exp(- \kappa t)\Big).$$
\end{theorem}

It is clear that Theorem \ref{thm:distance:independent'} and Theorem \ref{thm:distance:error} together imply Theorem \ref{thm:distance:dependent}. It remains to justify Theorem \ref{thm:distance:error}.

\section{Proof of  Theorem \ref{thm:distance:error}}\label{section:distance:dependent:2} 
We will view $P$ as the matrix of the first $n$ rows of a symmetric matrix $A$ of size $N\times N$. We will first need a standard deviation lemma (see for instance \cite{RV-HW}).

\begin{lemma}[Hanson-Wright inequality]\label{lemma:HW} There exists a constant $C=C(K_0)$ depending on the sub-gaussian moment such that the following holds. 

\begin{enumerate}[(i)]

\item Let $A$ be a fixed $M\times M$ matrix. Consider a random vector $\Bx=(x_1,\dots,x_{M})$ where the entries are i.i.d. sub-gaussian of mean zero and variance one. Then

$$\P(|\Bx^T A \Bx-\E \Bx^T A \Bx|>t) \le 2\exp(-C\min (\frac{t^2}{\|A\|^2_{HS}}, \frac{t}{\|A\|_2})).$$

In particularly, for any $t>0$

$$\P\left (|\Bx^T A \Bx-\E \Bx^T A \Bx|> t \|A\|_{HS} \right) \le \exp(-Ct).$$

\item Let $A$ be a fixed $N \times M$ matrix. Consider a random vector $\Bx=(x_1,\dots,x_{M})$ where the entries are i.i.d. sub-gaussian of mean zero and variance one. Then

$$\P\left (|\|A \Bx\|_2 -\|A\|_{HS}|> t\|A\|_2 \right) \le \exp(- Ct^2).$$

\end{enumerate}

\end{lemma}

Here the Hilbert-Schmidt norm of $A$ is defined as 

$$\|A\|_{HS} = \sqrt{\sum_{i,j}a_{ij}^2}.$$

By the first point of Lemma \ref{lemma:HW},

\begin{equation}\label{eqn:main-diag}
\P\left (|\Bz^T (PP^T)^{-1} \Bz -\tr(PP^T)^{-1}| > t \|(PP^T)^{-1}\|_{HS} \right) \le \exp(- Ct) 
\end{equation}

and

\begin{equation}\label{eqn:off-diag}
\P\left (|\Bz^T (PP^T)^{-1}P {\Bx^{[1]}}^T - \sum_i ((PP^T)^{-1}P)_{ii}| \ge t \|(PP^T)^{-1}P\|_{HS} \right ) \le  \exp(-Ct) .
\end{equation}

\subsection{Treatment for $\Bz (PP^T)^{-1}\Bz^T$.} To estimate \eqref{eqn:main-diag}, we need to consider the stable rank of the matrix $(PP^T)^{-1}$. By Theorem \ref{thm:least} (where $B$ is replaced by $P$) and by definition, we have the following important estimate with very high probability

\begin{equation}\label{eqn:sigma_n(P )}
\P(\sigma_n^{-2}(P) \le (\kappa' \eps)^{-2} N/m^2) \ge 1-\eps^{-\delta m} -\exp(-N^{\kappa'}).
\end{equation}

We next introduce another useful ingredient whose proof is deferred to Appendix \ref{section:interval}.

\begin{lemma}\label{lemma:manyvalues'} Let $C_1<C_2$ be given constants, then there exists a constant $C=C(C_1,C_2)$ such that the following holds with probability at least $1-N^{-\omega(1)}$: for any $\log^{C} N \le m \le N/C$, the intervals $[x_0+C_1m/N^{1/2},x_0+C_2m/N^{1/2}],0<x_0\le N/Cm$ contain $\Theta(m)$ singular values of $P$.
\end{lemma}

To avoid notational complication, we will assume
\begin{equation}\label{eqn:p(m,N)}
p= N^{-\omega(1)}+ \exp(-\delta m) + \exp(-N^{\kappa'}),
\end{equation}
where the exponent $\omega(1)$ might vary in different contexts.

We deduce the following asymptotic behavior of the trace of $(PP^T)^{-1}$.

\begin{corollary}\label{cor:HSnorm} With probability at least $1-p$,  

$$\|(PP^T)^{-1}P\|_{HS}^2 = \tr ((PP^T)^{-1})  \asymp N/m .$$
\end{corollary}

\begin{proof} For the lower bound, let $\sigma_i,\dots, \sigma_j$ be the  $\Theta(m)$ singular values of $P$ lying in the interval $[m/N^{1/2},Cm/N^{1/2}]$. Then 

$$ \tr((PP^T)^{-1}) \ge \sum_{i\le k\le j} \sigma_i^{-2} \gg m N/m^2 = N/m.$$

For the upper bound, first of all by \eqref{eqn:sigma_n(P )}, it suffices to assume that all of the singular values of $P$ are at least $cm/N^{1/2}$ for some sufficiently small $c$.

First, it is clear that the inverval $[cm/N^{1/2}, C_1 m/N^{1/2}]$ contains $O(m)$ singular values of $P$.  
For the remaining interval $[C_1m/N^{1/2},\Theta(N/N^{1/2})]$ (where we note that the top singular value of $P$ has order $\Theta(N^{1/2})$ with probability $1-\exp(-CN)$) we divide it into intervals $I_k$ of length $(C_2-C_1)m/N^{1/2}$ each with $1\le k\le \Theta(N/m)$. By Lemma \ref{lemma:manyvalues'} the number of singular vectors in each interval is proportional to $m$. As such

$$\tr((PP^T)^{-1})=\sum_{i} \sigma_i^{-2}  = \sum_k \sum_{i\in I_k} \sigma_i^{-2} \ll \sum_k m N/k^2m^2 \ll N/m.$$

\end{proof}

By Corollary  \ref{cor:HSnorm} and \eqref{eqn:sigma_n(P )},

\begin{equation}\label{eqn:main-diag'}
\P(\frac{\tr(PP^T)^{-1}}{\|(PP^T)^{-1}\|_2} \gg m) \ge 1- p.
\end{equation}

Notice that 

\begin{align*}
\|(PP^T)^{-1}\|_{HS}=\sqrt{\sum_i \sigma_i^2((PP^T)^{-1}) } & \le \sqrt{\|(PP^T)^{-1}\|_2 \sum_i \sigma_i((PP^T)^{-1})}\\ 
&= \sqrt{\|(PP^T)^{-1}\|_2} \sqrt{\tr((PP^T)^{-1})} .
\end{align*}

As a consequence, with probability at least $1-2p$

\begin{equation}\label{eqn:main-diag''} 
\frac{\tr((PP^T)^{-1})}{\|(PP^T)^{-1}\|_{HS}} \ge \frac{\sqrt{\tr((PP^T)^{-1})}}{\sqrt{\|(PP^T)^{-1}\|_2}}  \gg m^{1/2},
\end{equation}

where we used \eqref{eqn:main-diag'} in the last estimate.

Combining \eqref{eqn:main-diag},  \eqref{eqn:main-diag'} and  \eqref{eqn:main-diag''},  we have learned that

\begin{lemma}\label{lemma:denominator} There exists a positive constant $C$ such that the following holds for any $t>0$,

$$\P\Big(|\Bz^T (PP^T)^{-1} \Bz - \tr((PP^T)^{-1})|\ge \frac{Ct}{m^{1/2}} \tr((PP^T)^{-1})\Big) \le p+ \exp(-t/C).$$

Consequently, with probability at least $1-(p+\exp(-c \sqrt{m}))$

$$\Bz^T (PP^T)^{-1} \Bz \asymp \tr((PP^T)^{-1}) \asymp \frac{N}{m}.$$
\end{lemma}


\subsection{Treatment for $\Bz^T (PP^T)^{-1}P\Bx^{[1]}$} To show this quantity small, we will divide the treatment into two main steps.

{\bf Step 1: Set-up.}  We will present here the calculation for $((PP^T)^{-1}P)_{11}$, the formula for other $((PP^T)^{-1}P)_{ii}$ follows the same line. We first recall the definition of $\By_i$ at the beginning of Section \ref{section:distance:dependent:1}. For simplicity, we will drop the super index $[1]$ in all $\By_i$ (as we can view $P$ as the submatrix of the first $n$ rows of $A$). One interprets 

$$((PP^T)^{-1}P)_{11} = \langle  \row_1((PP^T)^{-1}),\Bc_1(P ) \rangle.$$

Because the matrix $PP^T$ depends on $\col_1(P )$, we need to separate the dependences. First of all, we write $PP^T$ as a rank-one perturbation $PP^T =  QQ^T + \col_1(P ) \col_1^T (P )$, where $Q$ is the matrix obtained from $P$ by removing its first column. By applying the formula for inverting rank-one perturbed matrices (which will be reused many times)

\begin{align}\label{eqn:rank-one}
(PP^T)^{-1} & =( QQ^T + \col_1(P ) \col_1^T (P ))^{-1} \nonumber \\ 
&= (QQ^T)^{-1}-\frac{ [(QQ^T)^{-1} \col_1(P )][ (\col_1(P ))^T  ((QQ)^T )^{-1}]}{ 1+  (\col_1(P ))^T  (QQ^T )^{-1} \col_1(P )}. 
\end{align}

It follows that 

\begin{align}\label{eqn:PP*:general}
(PP^T)^{-1} \col_1(P ) &= (QQ^T)^{-1} \col_1(P )-\frac{1}{ 1+  (\col_1(P ))^T  (QQ^T )^{-1} \col_1(P )} [(QQ^T)^{-1} \col_1(P )][ (\col_1(P ))^T  QQ^T )^{-1}] \col_1( P) \nonumber \\
&=  \frac{1}{ 1+  (\col_1(P ))^T  (QQ^T )^{-1} \col_1(P )}  (QQ^T)^{-1} \col_1(P ).
\end{align}

In particular,

$$((PP^T)^{-1}P)_{11} = \langle  \row_1((PP^T)^{-1}),\Bc_1(P ) \rangle =  \frac{1}{ 1+  (\col_1(P ))^T  (QQ^T )^{-1} \col_1(P )}  \langle \row_1((QQ^T)^{-1}), \col_1( P) \rangle .$$

Now the matrix $QQ^T$ still depends on $\col_1(P )$, so we are going to remove the dependence once more, this time using Schur complement formula (see \cite{HJ}).  

\begin{fact} Let  $M= \begin{pmatrix} X & Y \\ Y^T & Z \end{pmatrix}$. Assuming invertibility whenever necessary, we have 

$$ M^{-1}= \begin{pmatrix} (X -Y Z^{-1} Y^T)^{-1}& -X^{-1} Y (Z- Y^T X^{-1} Y)^{-1}  \\ (-X^{-1} Y (Z- Y^T X^{-1} Y)^{-1})^T  & (Z- Y^T X^{-1} Y)^{-1}\end{pmatrix}.$$
\end{fact}

Note that $QQ^T$ can be written as $\begin{pmatrix} \By_1^{[1]} (\By_1^{[1]})^T & \By_1^{[1]} R^T \\ R (\By_1^{[1]})^T & RR^T \end{pmatrix}$, where $R$ is the matrix obtained from $Q$ by removing its first row $\By_1^{[1]}$. From now on, for short we set 

$$x^2 := \By_1^{[1]} (\By_1^{[1]})^T$$

and

$$d^2:=x^2-  \By_1^{[1]} R^T   (RR^T )^{-1} R (\By_1^{[1]})^T.$$

To begin with, the top left corner is

\begin{align*}
(X -Y Z^{-1} Y^T)^{-1} &= (x^2 -  \By_1^{[1]} R^T  (RR^T)^{-1}R (\By_1^{[1]})^T)^{-1} = d^{-2}.
\end{align*}

Next, the bottom right Schur complement then can be calculated as 

$$(Z- Y^T X^{-1} Y)^{-1} = [RR^T - \frac{1}{x^2} R (\By_1^{[1]})^T  \By_1^{[1]}R^T]^{-1}.$$

Note that this again can be considered as rank-one perturbation as in \eqref{eqn:rank-one}, 

\begin{align*}
[RR^T - \frac{1}{x^2} R (\By_1^{[1]})^T  \By_1^{[1]}R^T]^{-1}&= (RR^T)^{-1} +\frac{1}{x^2} \frac{ [(RR^T)^{-1} R (\By_1^{[1]})^T ][ \By_1^{[1]} R^T (RR^T )^{-1}]}{ 1- \frac{1}{x^2} \By_1^{[1]} R^T   (RR^T )^{-1} R (\By_1^{[1]})^T }\\
&=  (RR^T)^{-1} +\frac{ [(RR^T)^{-1} R (\By_1^{[1]})^T ][ \By_1^{[1]} R^T (RR^T )^{-1}]}{d^2 }.
\end{align*}

Similarly, the top right Schur complement is

\begin{align*}
-X^{-1} Y (Z- Y^T X^{-1} Y)^{-1} &= - \frac{1}{x^2}   \By_1^{[1]} R^T \left [(RR^T)^{-1} +\frac{ [(RR^T)^{-1} R (\By_1^{[1]})^T ][ \By_1^{[1]} R^T (RR^T )^{-1}]}{d^2}\right]\\
&= - \frac{1}{x^2}   \By_1^{[1]} R^T (RR^T)^{-1} +(\frac{1}{x^2} -\frac{1}{d^2})  \By_1^{[1]} R^T (RR^T)^{-1}\\
&=  -d^{-2}   \By_1^{[1]} R^T (RR^T)^{-1}.
\end{align*}

Putting together,

\begin{equation}\label{eqn:QQ^T}
(QQ^T)^{-1} = \begin{pmatrix}d^{-2} & - d^{-2} \By_1^{[1]} R^T (RR^T)^{-1} \\ - d^{-2}  (RR^T)^{-1}  R (\By_1^{[1]})^{T}
 & (RR^T)^{-1} +d^{-2}[(RR^T)^{-1} R (\By_1^{[1]})^T ][ \By_1^{[1]} R^T (RR^T )^{-1}] \end{pmatrix}
 \end{equation}

where we note that the involved matrices are invertible by Theorem \ref{thm:singularity} with extremely large probability.


It follows that

\begin{align*}
\langle \row_1((QQ^T)^{-1}), \col_1( P) \rangle = d^{-2} [\col_1(P )_1 - \By_1^{[1]} R^T (RR^T)^{-1} \col_1(P )^{[1]}].
\end{align*}

Also, the denominator of \eqref{eqn:PP*:general} can be written as

\begin{align*}
 (\col_1(P ))^T  (QQ^T )^{-1} \col_1(P ) &=  d^{-2}[(\col_1(P )_1)^2 -  2\col_1(P )_1 (\col_1(P )^{[1]})^T (RR^T)^{-1} R (\By_1^{[1]})] \\
&+ (\col_1(P )^{[1]})^T (RR^T)^{-1} \col_1(P )^{[1]} \\
&+ d^{-2} [(\col_1(P )^{[1]})^T (RR^T)^{-1} R (\By_1^{[1]})^T ][ \By_1^{[1]} R^T (RR^T )^{-1} \col_1(P )^{[1]}]\\
&= (\col_1(P )^{[1]})^T (RR^T)^{-1} \col_1(P )^{[1]} + d^{-2}[\col_1(P )_1 - \By_1^{[1]} R^T (RR^T)^{-1} \col_1(P )^{[1]}]^2.
\end{align*}

Combining the formulas, we arrive at the following.

\begin{lemma}\label{lemma:formula:diagonal}
We have 

$$((PP^T)^{-1}P)_{11} =\frac{\col_1(P )_1 - \By_1^{[1]} R^T (RR^T)^{-1} \col_1(P )^{[1]}}{d^2 ( 1+(\col_1(P )^{[1]})^T (RR^T)^{-1} \col_1(P )^{[1]})  +  (\col_1(P )_1 - \By_1^{[1]} R^T (RR^T)^{-1} \col_1(P )^{[1]})^2 }.$$
\end{lemma}

To proceed further, note that $d^2=x^2 -\By_1^{[1]} R^T   (RR^T )^{-1} R (\By_1^{[1]})^T$ is just the distance from $\By_1^{[1]}$ to the subspace generated by the rows of $R$, and thus is well concentrated around $(N-1)-(n-1) = m$ by Theorem \ref{thm:distance:independent}, that is with probability at least $1-\exp(-cm)$

\begin{equation}\label{eqn:d}
d^2 \gg m.
\end{equation}

\begin{corollary}\label{cor:estimate:1}
With probability at least $1-(p+\exp(-c\sqrt{m}))$,

$$|((PP^T)^{-1}P)_{11}| = O(\frac{1}{\sqrt{N}}).$$

Consequently,

$$\sum_i (PP^T)^{-1}P)_{ii} = O(\frac{n}{\sqrt{N}}) = O(\sqrt{N}).$$
\end{corollary}

\begin{proof} By Cauchy-Schwarz, 
\begin{align*} 
&\frac{|\col_1(P )_1 - \By_1^{[1]} R^T (RR^T)^{-1} \col_1(P )^{[1]}|}{d^2 (1+(\col_1(P )^{[1]})^T (RR^T)^{-1} \col_1(P )^{[1]})  +  (\col_1(P )_1 - \By_1^{[1]} R^T (RR^T)^{-1} \col_1(P )^{[1]})^2 }\\
&\le \frac{1}{2\sqrt{ d^2 (1+(\col_1(P )^{[1]})^T (RR^T)^{-1} \col_1(P )^{[1]})}} =O (\frac{1}{\sqrt{N}}),
\end{align*} 
where in the last estimate we used $d^2 \gg m$ and Lemma \ref{lemma:denominator} for $R$. 
\end{proof}

We also deduce the following consequence.

\begin{corollary}\label{cor:estimate:2}
With probability at least $1-p-\exp(-c\sqrt{m})$

$$|\col_1(P )_1 - \By_1^{[1]} R^T (RR^T)^{-1} \col_1(P )^{[1]}| \ll \sqrt{N}.$$
\end{corollary}

\begin{proof} By Corollary \ref{cor:estimate:1}, with probability at least $1-p-\exp(c\sqrt{m})$ 

$$|\sum_i (R^T (RR^T)^{-1} )_{ii}|=O(\sqrt{N}).$$

Also, by Corollary \ref{cor:HSnorm}, with probability at least $1-p$

$$\|R^T (RR^T)^{-1}\|_{HS}^2 = \tr((RR^T)^{-1})=O(\frac{N}{m}).$$

Conditioning on these events of $R$, Lemma \ref{lemma:HW} applied to the random vectors $\By_1^{[1]}$ and  $\col_1(P )^{[1]}$  implies that with probability at least $1-\exp(-c\sqrt{m})$ 

$$|\By_1^{[1]} R^T (RR^T)^{-1} \col_1(P )^{[1]} - \sum_i (R^T (RR^T)^{-1} )_{ii}|\le \sqrt{m} \|R^T (RR^T)^{-1}\|_{HS} \ll \sqrt{N}.$$

The proof is complete by noting that due to the sub-gaussian assumption, $|\col_1(P )_1| =O(\sqrt{N})$ with probability $1- \exp(-\Theta(N))$. 

\end{proof}

We remark that Lemma \ref{lemma:denominator} and Corollary \ref{cor:estimate:2} allow us to conclude that $\frac{[x_1 - \Bz^T (PP^T)^{-1}P {\Bx^{[1]}}^T]^2}{1+ \Bz^T (PP^T)^{-1} \Bz}$ has order at most $m$ with high probability, but this is not strong enough for Theorem \ref{thm:distance:error}. We will improve this in the next phase of the proof.

{\bf Step II: Comparison.}  In this step we show that $\sum_{1\le i\le n}((PP^T)^{-1}P)_{ii}$ is close to $\sum_{1\le i\le n-1}((RR^T)^{-1}R)_{ii}$.

Recall that
\begin{align*}
(PP^T)^{-1} & =( QQ^T + \Bw \Bw^T)^{-1}  \\ 
&= (QQ^T)^{-1}-\frac{1}{ 1+  \Bw^T  (QQ^T )^{-1} \Bw} [(QQ^T)^{-1} \Bw][ \Bw^T  (QQ^T )^{-1}]\\
&=:(QQ^T)^{-1}-Q'
\end{align*}

where for convenience, we denote the second matrix by $Q'$. 

We recall the formula \eqref{eqn:QQ^T} for $(QQ^T)^{-1}$ and 

$$P = (\begin{matrix} \Bw & Q \end{matrix}) = \begin{pmatrix} x_0 & \By \\ \Bz & R\end{pmatrix}$$

where for short we denote  $\Bw= \col_1( P)=(x_0,\Bz)^T$ and 

$$d^2=  \By \By^T-  \By R^T   (RR^T )^{-1} R \By^T .$$

We now compute $\sum_{2\le i\le n} ((PP^T)^{-1}P)_{ii}$. To do this, we start from $(QQ^T)^{-1} P$ and eliminate its first row and column to obtain a matrix $M_1$ of size $(n-1)\times (N-1)$

\begin{align}\label{eqn:M_1}
M_1&=-d^{-2} (RR^T)^{-1}  R \By^{T}  \By+\Big ((RR^T)^{-1} +d^{-2} [(RR^T)^{-1} R \By^T ][ \By R^T (RR^T )^{-1}] \Big )R \nonumber \\
&=(RR^T)^{-1}R - d^{-2} (RR^T)^{-1}  R \By^{T}  \By (I-R^T (RR^T )^{-1}R)\nonumber \\
&=(RR^T)^{-1}R -M_1',
\end{align}

with 

\begin{equation}\label{eqn:M_1'}
M_1':=d^{-2} (RR^T)^{-1}  R \By^{T}  \By (I-R^T (RR^T )^{-1}R).
\end{equation}

Next, for the contribution of $Q'P$ (after the elimination of its first row and column), we need to compute the vector $(QQ^T)^{-1} \Bw$.

\begin{align*}
(QQ^T)^{-1} \Bw &= \begin{pmatrix} d^{-2} (x_0-\By R^T (RR^T)^{-1}  \Bz)  \\ -x_0 d^{-2} (RR^T)^{-1}  R \By^{T}  + \Big(  (RR^T)^{-1} +d^{-2} [(RR^T)^{-1} R \By^T ][ \By R^T (RR^T )^{-1}] \Big) \Bz\end{pmatrix}\\
&= \begin{pmatrix} d^{-2} (x_0-\By R^T (RR^T)^{-1}  \Bz)  \\ -d^{-2}(x_0-\By R^T (RR^T)^{-1}  \Bz) (RR^T)^{-1}  R \By^{T}  + (RR^T)^{-1} \Bz \end{pmatrix}.\\ 
&= \begin{pmatrix} a \\ -a (RR^T)^{-1}  R \By^{T} + (RR^T)^{-1} \Bz \end{pmatrix},
\end{align*}

with

$$a:= d^{-2} (x_0-\By R^T (RR^T)^{-1}  \Bz).$$

As a result,  the bottom left submatrix of $(QQ^T)^{-1} \Bw  \Bw^T  (QQ^T )^{-1}$ is the vector $-a^2 (RR^T)^{-1}  R \By^{T} + a(RR^T)^{-1} \Bz  $ and the bottom right submatrix is the matrix

$$a^2 (RR^T)^{-1}  R \By^{T} \By R^T (RR^T)^{-1} + (RR^T)^{-1} \Bz  \Bz^T (RR^T)^{-1} - a (RR^T)^{-1}  R \By^{T} \Bz^T (RR^T)^{-1}- a (RR^T)^{-1}\Bz \By R^T   (RR^T)^{-1}.$$

It follows that the matrix $M_2$ obtained by eliminating the first row and column of $Q'P$ can be written as

\begin{align}\label{eqn:M_2}
M_2 &=\frac{1}{1+  \Bw^T  (QQ^T )^{-1} \Bw}\Big[ -a^2 (RR^T)^{-1}  R \By^{T}\By + a(RR^T)^{-1} \Bz  \By \nonumber \\
&+ a^2 (RR^T)^{-1}  R \By^{T} \By R^T (RR^T)^{-1} R +  (RR^T)^{-1} \Bz  \Bz^T (RR^T)^{-1}R - a (RR^T)^{-1}  [R \By^{T} \Bz^T + \Bz \By R^T](RR^T)^{-1}R \Big] \nonumber \\
&= \frac{1}{1+  \Bw^T  (QQ^T )^{-1} \Bw}\Big[ a^2 (RR^T)^{-1}  R \By^{T} \By (R^T (RR^T)^{-1} R-I) + a(RR^T)^{-1} \Bz  \By \nonumber \\ 
&+   (RR^T)^{-1} \Bz  \Bz^T (RR^T)^{-1}R -  a (RR^T)^{-1}  [R \By^{T} \Bz^T + \Bz \By R^T] (RR^T)^{-1}R \Big].
\end{align}

In summary, 

$$\sum_{2\le i\le n} ((PP^T)^{-1}P)_{ii} = \sum_{1\le i\le n-1} ((RR^T)^{-1}R)_{ii} - \sum_{1\le i\le n-1}(M_1'+M_2)_{ii}.$$

In what follows, by using the formulas for $M_1',M_2$ from \eqref{eqn:M_1}, \eqref{eqn:M_2} we show that $\sum_{1\le i\le n-1}(M_1'+M_2)_{ii}$ is negligible. 

We will try to simplify the formulae a bit. First,

\begin{align}\label{eqn:wQw}
1+ \Bw^T  (QQ^T )^{-1} \Bw &= 1+ (\col_1(P ))^T  (QQ^T )^{-1} \col_1(P ) =1+ (\col_1(P )^{[1]})^T (RR^T)^{-1} \col_1(P )^{[1]} \nonumber  \\
&+  \frac{1}{x^2 -\By_1^{[1]} R^T   (RR^T )^{-1} R (\By_1^{[1]})^T } [\col_1(P )_1 - \By_1^{[1]} R^T (RR^T)^{-1} \col_1(P )^{[1]}]^2 \nonumber \\
&= 1+\Bz^T (RR^T)^{-1} \Bz + d^{-2} [x_0 - \By R^T (RR^T)^{-1} \Bz]^2.
\end{align}

Thus, by Lemma \ref{lemma:denominator} together with \eqref{eqn:d} and Corollary \ref{cor:estimate:2}, with probability at least $1-(p + \exp(-c\sqrt{m}))$,

\begin{equation}\label{eqn:wQw'}
\Bw^T  (QQ^T )^{-1} \Bw \asymp \tr((RR^T)^{-1}) \asymp \frac{N}{m}.
\end{equation}

Also from \eqref{eqn:M_1'} and \eqref{eqn:wQw}

\begin{align*}
(1+\Bw^T  (QQ^T )^{-1} \Bw)M_1' &= d^{-2} (RR^T)^{-1}  R \By^{T}  \By (I-R^T (RR^T )^{-1}R)(1+ \Bz^T (RR^T)^{-1} \Bz)+\\  
&+a^2  (RR^T)^{-1}  R \By^{T}  \By (I-R^T (RR^T )^{-1}R).
\end{align*}


Hence the normalized matrix $(1+\Bw^T  (QQ^T )^{-1} \Bw)(M_1'+M_2)$ can be expressed as 
\begin{align*}
(1+\Bw^T  (QQ^T )^{-1} \Bw)(M_1'+M_2)&=((RR^T)^{-1}) \Big( d^{-2}   (1+\Bz^T (RR^T)^{-1} \Bz ) R \By^{T}  \By (I-R^T (RR^T )^{-1}R)+ a\Bz  \By\\ 
&+  \Bz  \Bz^T (RR^T)^{-1}R - a  [R \By^{T} \Bz^T + \Bz \By R^T] (RR^T)^{-1}R\Big)\\
&:= (RR^T)^{-1} S.
\end{align*}


We can write the second matrix $S$ as $\sum_1+\sum_2$ where
\begin{align*}
\sum_1:&=d^{-2} \Big( \big[(x_0 - \By R^T (RR^T)^{-1}\Bz) \Bz + (1+\Bz^T (RR^T)^{-1} \Bz) R \By^T \big] \big [\By (I-R^T (RR^T )^{-1}R)\big]\Big),
\end{align*}
and
\begin{align*}
\sum_2&:=d^{-2} \Big(\big[(\By (I-R^T (RR^T )^{-1}R) \By^T) \Bz - (x_0 - \By R^T (RR^T)^{-1} \Bz)) R \By^{T}\big]  \big [\Bz^T  (RR^T)^{-1}R\big] \Big) .
\end{align*}




By the triangle inequality,

\begin{align}\label{eqn:M12}
&\ |\sum_{1\le i\le n-1}(M_1'+M_2)_{ii}| \le \nonumber \\
&\le \frac{1}{d^2(1+\Bw^T  (QQ^T )^{-1} \Bw)} |x_0 - \By R^T (RR^T)^{-1}\Bz| |\sum_{1\le i\le n-1} \big( (RR^T )^{-1} \Bz \By (I-R^T (RR^T )^{-1}R)\big)_{ii}| \nonumber \\
&+  \frac{1}{d^2(1+\Bw^T  (QQ^T )^{-1} \Bw)}(1+\Bz^T (RR^T)^{-1} \Bz)  |\sum_{1\le i\le n-1} \big( (RR^T )^{-1} R \By^T  \By (I-R^T (RR^T )^{-1}R)\big)_{ii}| \nonumber \\
&+ \frac{1}{d^2(1+\Bw^T  (QQ^T )^{-1} \Bw)}(\By (I-R^T (RR^T )^{-1}R) \By^T) |\sum_{1\le i\le n-1} \big( (RR^T )^{-1} \Bz \Bz^T  (RR^T)^{-1}R\big)_{ii}|\nonumber  \\
&-  \frac{1}{d^2(1+\Bw^T  (QQ^T )^{-1} \Bw)}  |x_0 - \By R^T (RR^T)^{-1} \Bz| |\sum_{1\le i\le n-1} \big( (RR^T )^{-1} R \By^{T}\Bz^T  (RR^T)^{-1}R\big)_{ii}| \nonumber \\
&:= E_1+E_2+E_3+E_4.
\end{align}

To complete our estimates for $E_1,E_2,E_3,E_4$, we recall from   \eqref{eqn:main-diag''}, \eqref{eqn:d} and Corollary \ref{cor:estimate:2} that with probability at least $1-(2p+\exp(c\sqrt{m}))$

\begin{itemize}
\item 
\begin{equation}\label{eqn:altogether2}
m^{1/2} \|(RR^T)^{-1}\|_{HS} \ll \tr((RR^T)^{-1}),
\end{equation}
\item 
\begin{equation}\label{eqn:d'}
m\ll d^2,
\end{equation}
\item 
\begin{equation}\label{eqn:altogether1}
|x_0 - \By R^T (RR^T)^{-1}\Bz| \ll \sqrt{N}.
\end{equation}
\end{itemize}

In what follows we will be conditioning on these events. To proceed further, we will repeatedly use the following elementary (Cauchy-Schwarz) fact that if $\Ba=(a_1,\dots,a_{m_1}) \in \R^{m_1}$ and $\Bb=(b_1,\dots,b_{m_2})\in \R^{m_2}$ are column vectors with $m_1\le m_2$ then 

$$\sum_{1\le i\le m_1} (\Ba \Bb^T)_{ii} = \sum_{1\le i\le m_1} a_i b_i \le \|\Ba\|_2 \|\Bb\|_2.$$

By (2) of Theorem \ref{lemma:HW}, the following holds with probability at least $1- O(p+ \exp(-c\sqrt{m}))$

\begin{align*}
|E_1| &= \frac{1}{d^2(1+\Bw^T  (QQ^T )^{-1} \Bw)} |x_0 - \By R^T (RR^T)^{-1}\Bz| |\sum_{1\le i\le n-1} \big( (RR^T )^{-1} \Bz \By (I-R^T (RR^T )^{-1}R)\big)_{ii}| \\
& \ll (\tr((RR^T)^{-1}))^{-1} m^{-1} \sqrt{N} \|(RR^T)^{-1}\|_{HS}  \|(I-R^T (RR^T )^{-1}R)\|_{HS}\\
&=   (\tr((RR^T)^{-1}))^{-1}  m^{-1}\sqrt{N} \|(RR^T)^{-1}\|_{HS}  \sqrt{\tr((I-R^T (RR^T )^{-1}R)^2)}\\
&=  (\tr((RR^T)^{-1}))^{-1}  m^{-1}\sqrt{N} \|(RR^T)^{-1}\|_{HS}  \sqrt{\tr(I-R^T (RR^T )^{-1}R)}\\
&\ll (\tr((RR^T)^{-1}))^{-1}  m^{-1} \sqrt{N} \sqrt{m} \|(RR^T)^{-1}\|_{HS} \ll \frac{\sqrt{N}}{m},
\end{align*}

where in the second to last inequality we used the fact that with probability at least $1-\exp(-N^\kappa)$ (see also the discussion by Theorem \ref{thm:distance:independent'} and by \eqref{eqn:d}) the projection matrix $I-R^T (RR^T )^{-1}R$ onto the orthogonal complement of the subspace generated by the row vectors of $R$ in $\R^{N-2}$ has trace $N-2-(n-2)=m$.

Similarly, with probability at least $1- O(p+ \exp(-c\sqrt{m}))$

\begin{align*}
E_2=&  \frac{1}{d^2(1+\Bw^T  (QQ^T )^{-1} \Bw)}(1+\Bz^T (RR^T)^{-1} \Bz)  |\sum_{1\le i\le n-1} \big( (RR^T )^{-1} R \By^T  \By (I-R^T (RR^T )^{-1}R)\big)_{ii}| \nonumber \\
& \ll (\tr((RR^T)^{-1}))^{-1} m^{-1} \tr((RR^T)^{-1})  \|(RR^T)^{-1}R\|_{HS}  \|(I-R^T (RR^T )^{-1}R)\|_{HS}\\
&\ll m^{-1} \sqrt{\frac{N}{m}} \sqrt{m} = \frac{\sqrt{N}}{m};
\end{align*}

and 

\begin{align*}
E_3 &= \frac{1}{d^2(1+\Bw^T  (QQ^T )^{-1} \Bw)}(\By (I-R^T (RR^T )^{-1}R) \By^T) |\sum_{1\le i\le n-1} \big( (RR^T )^{-1} \Bz \Bz^T  (RR^T)^{-1}R\big)_{ii}|\nonumber  \\
& \le (\tr((RR^T)^{-1}))^{-1} m^{-1} m \|(RR^T)^{-1} \|_{HS}  \|(RR^T )^{-1}R)\|_{HS}\\
&\ll   (\tr((RR^T)^{-1}))^{-1/2} \|(RR^T)^{-1}\|_{HS} \ll   (\tr((RR^T)^{-1}))^{-1/2} \tr((RR^T)^{-1}) m^{-1/2} \\
&\ll  (\tr((RR^T)^{-1}))^{1/2} m^{-1/2} \ll  \frac{\sqrt{N}}{m}.
\end{align*}

Lastly, with probability at least $1- O(p+ \exp(-c\sqrt{m}))$ we also have

\begin{align*}
|E_4|&= \frac{1}{d^2(1+\Bw^T  (QQ^T )^{-1} \Bw)}  |x_0 - \By R^T (RR^T)^{-1} \Bz| |\sum_{1\le i\le n-1} \big( (RR^T )^{-1} R \By^{T}\Bz^T  (RR^T)^{-1}R\big)_{ii} |\nonumber \\
& \ll (\tr((RR^T)^{-1}))^{-1} m^{-1} \sqrt{N} \|(RR^T)^{-1}R\|_{HS}^2 \\
&\ll \frac{\sqrt{N}}{m}.
\end{align*}

We sum up below

\begin{align*}
& |\sum_{1\le i\le n}((PP^T)^{-1}P)_{ii} - \sum_{1\le i\le n-1}((RR^T)^{-1}R)_{ii} |\\
&\le  |((PP^T)^{-1}P)_{11}|+ | \sum_{2\le i\le n}((PP^T)^{-1}P)_{ii} - \sum_{1\le i\le n-1}((RR^T)^{-1}R)_{ii} |\\
&\le |((PP^T)^{-1}P)_{11}| + |E_1|+E_2+E_3 + |E_4|\\
&\ll \frac{1}{\sqrt{N}} + \frac{\sqrt{N}}{m} \ll  \frac{\sqrt{N}}{m}.
\end{align*}

\begin{lemma}\label{lemma:comparison:trace}
With probability at least $1- O(p+ \exp(-c\sqrt{m}))$ we have 
$$|\CE_1| \ll  \frac{\sqrt{N}}{m}$$
where $\CE_1= \sum_{1\le i\le n}((PP^T)^{-1}P)_{ii} - \sum_{1\le i\le n-1}((RR^T)^{-1}R)_{ii}$. 
\end{lemma}

Here it is emphasized that the index 1 of $\CE_1$ shows the error of comparison between $\sum_{1\le i\le n}((PP^T)^{-1}P)_{ii}$ and $\sum_{1\le i\le n-1}((RR^T)^{-1}R)_{ii}$ where $R$ is obtained by removing the {\it first} row and column of $P$. If we remove its $k$-th row and column instead, then we use $\CE_k$ to denote the difference.

\subsection{Putting things together.} Set 

$$T:=\sum_{1\le i\le n}((PP^T)^{-1}P)_{ii}.$$

Recall the formula of $((PP^T)^{-1}P)_{11}$ in Lemma  \ref{lemma:formula:diagonal}

$$((PP^T)^{-1}P)_{11} =\frac{\col_1(P )_1 - \By_1^{[1]} R^T (RR^T)^{-1} \col_1(P )^{[1]}}{d^2 ( 1+(\col_1(P )^{[1]})^T (RR^T)^{-1} \col_1(P )^{[1]})  +  (\col_1(P )_1 - \By_1^{[1]} R^T (RR^T)^{-1} \col_1(P )^{[1]})^2 }.$$

By the triangle inequality, the numerator can be written as

\begin{align*}
& \ \  \ \  \col_1(P )_1 - \By_1^{[1]} R^T (RR^T)^{-1} \col_1(P )^{[1]}\\ 
&=\col_1(P )_1 - \sum_{1\le i\le n-1}((RR^T)^{-1}R)_{ii} -  [\By_1^{[1]} R^T (RR^T)^{-1} \col_1(P )^{[1]} - \E_{\By_1^{[1]}} \By_1^{[1]} R^T (RR^T)^{-1} \col_1(P )^{[1]}] \\
&= \CE_1-T - [\By_1^{[1]} R^T (RR^T)^{-1} \col_1(P )^{[1]} - \E_{\By_1^{[1]}} \By_1^{[1]} R^T (RR^T)^{-1} \col_1(P )^{[1]}] +\col_1(P )_1.
\end{align*}

By Lemma \ref{lemma:HW}, with probability at least $1- O(p+ \exp(-c\sqrt{m}) +\exp(-Ct))$, this can be bounded from above by  $-T + O(\sqrt{N}/m) + t (\|(R_iR_i^T)^{-1}R_i\|_{HS}+1)$, and hence by Corollary \ref{cor:HSnorm}

$$\col_1(P )_1- \By_1^{[1]} R^T (RR^T)^{-1} \col_1(P )^{[1]} \le -T + O(\sqrt{N}/m + t (\sqrt{N/m}+1)).$$

Furthermore,  by Lemma \ref{lemma:denominator} and \eqref{eqn:d'}, with probability at least $ 1-O(p+\exp(-c\sqrt{m}))$ the denominator can be estimated from below by

\begin{align*}
D_1&=\Big[x^2 -\By_1^{[1]} R^T   (RR^T )^{-1} R (\By_1^{[1]})^T \Big] \Big [ 1+(\col_1(P )^{[1]})^T (RR^T)^{-1} \col_1(P )^{[1]}\Big]  \\
&+  [\col_1(P )_1 - \By_1^{[1]} R^T (RR^T)^{-1} \col_1(P )^{[1]}]^2 \gg N.
\end{align*}

Thus with probability at least $1- O(p+ \exp(-c\sqrt{m}) +\exp(-Ct))$

$$((PP^T)^{-1}P)_{11}+ \frac{ T}{D_1} =O( \frac{\sqrt{N}/m + t (\sqrt{N/m}+1)}{N}).$$

Estimating similarly for $(PP^T)^{-1}P)_{ii}, i\ge 2$, by the triangle inequality and by taking union bound, the following holds with probability at least $1- O(Np+ N\exp(-c\sqrt{m}) +N\exp(-Ct))$

\begin{align*}
|T+ T(\sum_i \frac{1}{D_i})| &\le \sum_i  O( \frac{\sqrt{N}/m + t (\sqrt{N/m}+1)}{N}) \\ 
&=O(\sqrt{N}/m + t (\sqrt{N/m}+1)).\\ 
\end{align*}

We summarize into a lemma as follows.

\begin{lemma}\label{lemma:PPP:diagsum}
With  probability at least $1- O(Np+ N\exp(-c\sqrt{m}) +N\exp(-Ct))$,

$$|T|=O(\sqrt{N}/m + t (\sqrt{N/m}+1)).$$
\end{lemma}

We now complete the proof of Theorem \ref{thm:distance:error}. By Hanson-Wright estimates, and by Lemma \ref{lemma:PPP:diagsum}, with probability at least $1- O(Np+ N\exp(-c\sqrt{m}) +N\exp(-Ct))$,

\begin{align*}
\frac{[x_1 - \Bz^T (PP^T)^{-1}P {\Bx^{[1]}}^T]^2}{1+ \Bz^T (PP^T)^{-1} \Bz} &\le \frac{\Big[|\sum_{1\le i\le n}((PP^T)^{-1}P)_{ii}|  + t \|(PP^T)^{-1}P\|_{HS}\Big]^2}{\tr(PP^T)} \\
&=O( \frac{(\sqrt{N}/m + t (\sqrt{N/m}+1))^2}{N/m}) \\
&=O( t^2)
\end{align*}

provided that $t$ is sufficiently large. Our proof is then complete by choosing $\kappa$ (stated in Theorem \ref{thm:distance:error}) to be any constant smaller than $\delta, \kappa', c$.

\section{Proof of Theorem \ref{thm:least}: sketch}\label{section:least:introduction}
In this section we sketch the idea to prove Theorem \ref{thm:least}, details of the proof will be presented in later sections. Roughly speaking, we will follow the treatment by Rudelson and Vershynin from \cite{RV-rec} and by Vershynin from \cite{V} with some modifications.   We also refer the reader to a more recent paper by Rudelson and Vershynin \cite{RV-gap} for similar treatments.

We first need some preparations, for a cosmetic reason, let us view $B$ as a {\it column} matrix of size $N$ by $n$ of the last $n$ columns of $A$ from now on,
$$B = \begin{pmatrix}\col_{m+1}(A) & \dots & \col_N(A)\end{pmatrix}.$$
Let $c_0, c_1 \in (0,1)$ be two numbers. We will choose their values later
as small constants that depend only on the subgaussian parameter $K_0$.
\begin{definition}
A vector $\Bx \in \R^n$ is called {\it sparse} if $|\supp(\Bx)| \le c_0 n$.
A vector $\Bx \in S^{n-1}$ is called {\it compressible} if $\Bx$ is within Euclidean distance $c_1$ from the set of all sparse vectors.
  A vector $\Bx \in S^{n-1}$ is called {\it incompressible} if it is not compressible.
  
  The sets of compressible and incompressible vectors in $S^{n-1}$
  will be denoted by $\Comp(c_0, c_1)$ and $\Incomp(c_0, c_1)$ respectively.
\end{definition}
  
Given a vector random variable $\Bx$ and a radius $r$, we define the {\it Levy concentration} of $\Bx$ (or the {\it small ball probability} with radius $r$) to be
$$\CL(\Bx,r):= \sup_{\Bu} \P(\|\Bx-\Bu\|_2 \le r).$$
In order to prove Theorem \ref{thm:least}, we decopose $S^{n-1}$ into compressible and incompressible vectors for some appropriately chosen parameter $c_0$ and $c_1$. Let $\CE_K$ be the event that 
$$\CE_K=\{\|B\|_2 \le 3K\sqrt{N}\}.$$
\begin{align*}
\P\big (\min_{\Bx\in S^{n-1}} \|B^T \Bx\|_2 \le \ep (\sqrt{N}-\sqrt{n}) \cap \CE_K\big) &\le \P\big (\min_{\Bx\in \Comp(c_0,c_1)} \|B^T \Bx\|_2 \le \ep (\sqrt{N}-\sqrt{n})\cap \CE_K \big)\\ 
&+ \P\big (\min_{\Bx\in \Incomp(c_0,c_1)} \|B^T \Bx\|_2 \le \ep (\sqrt{N}-\sqrt{n})\cap \CE_K\big).
\end{align*}
  
In this section we only treat with the compressible vectors by giving a stronger bound.  

First of all, we bound for a fixed vector $\Bx$. The following follows from the mentioned work by Vershynin for symmetric matrices.

\begin{lemma}\cite[Proposition 4.1]{V}\label{lemma:compressible:1} For every vector $\Bx\in S^{n-1}$ one has
$$\CL(B^T\Bx, c\sqrt{N})=\sup_{\Bu}\P(\|B^T\Bx-\Bu\|_2 \le c\sqrt{N}) \le \exp(-cn).$$  
\end{lemma} 
  




  
 
   
  
 
 
  



  
Now we bound uniformly over all compressible vectors (see also \cite[Proposition 4.2]{V}).

\begin{lemma}\label{lemma:compressible:2} We have
$$\P \Big\{ \inf_{\Bx/\|\Bx\|_2 \in \Comp(c_0,c_1)} \|B\Bx-\Bu\|_2 \le c \sqrt{N} \|\Bx\|_2 \cap \CE_K\Big\} \le  \exp(-c n/2).$$
\end{lemma}  
  
\begin{proof} We will sketch the proof. First, it is known that there exists a $(2c_1)$-net $\CN$ of the set $\Comp(c_0,c_1)$ such that 

$$\CN \le (9/c_0c_1)^{c_0n}.$$  

Next, by unfolding the vectors, there exists $\Bv_0$ such that if $\Bx/\|\Bx\|_2 \in \Comp(c_0,c_1)$ with $\|B\Bx-\Bu\|_2 \le c \sqrt{N} \|\Bx\|_2$ and assuming  $\CE_K$, then there exists $\Bx_0\in \CN$ such that $\|B\Bx_0 - \Bv_0\|_2 \le c\sqrt{N}$.

The proof is complete by taking union bound as 
$$(9/c_0c_1)^{c_0 n} \frac{20}{c_1} 2 \exp(-cn) \le \exp(-cn/2),$$
if $c_0$ is chosen small enough depending on $c_1$ and $c$.
\end{proof}

 \section{Proof of Theorem \ref{thm:least}: treatment for incompressible vectors}\label{section:distance:compressible}  
 
 Let $\Bx_1,\dots,\Bx_n \in \R^N$ denote the columns of the matrix $B$. Given a subset $J\subset [n]^{d}$, where $d=\delta m=\delta(N-n)$ for some sufficiently small $\delta$ to be chosen, we consider the subspace 
$$H_J:=span(\Bx_i)_{i\in J}.$$
Define 
$$\spread_J:= \Big \{\By \in S(\R^J): K_1/\sqrt{d} \le |y_k| \le K_2/\sqrt{d}, k\in J \Big \},$$
where $S(\R^J)$ is the unit sphere in the Euclidean space determined by the indices of $J$.  In what follows $J$ is a subset chosen randomly uniformly among the subsets of cardinality $d$ in $[n]$ and $P_J$ is the projection onto the coordinates indexed by $J$.

\begin{lemma}\cite[Lemma 6.1]{RV-rec}\label{lemma:incomp:1} For every $c_0,c_1\in (0,1)$, there exist $K_1,K_2,c>0$ which depend only on $c_0,c_1$ such that the following holds. For every $\Bx\in \Incomp(c_0,c_1)$, the event
$$E(\Bx):=\Big\{ P_J\Bx/\|P_J\Bx\|_2 \in \spread_J \mbox{ and } c_1\sqrt{d}/\sqrt{2N} \le \|P_J\Bx\|_2 \le \sqrt{d}/\sqrt{c_0 N} \Big\}$$
satisfies 
$$\P_J(E(\Bx))> c^{d}.$$
\end{lemma}

This lemma follows easily from a simple property of incompressible vectors whose proof is omitted.

\begin{claim}\cite[Lemma 2.5]{RV-rec}\label{claim:spread}
Let $\Bx \in \Incomp(c_0,c_1)$. Then there exists a set $\sigma=\sigma(\Bx)\subset [n]$ of cardinality $|\sigma|\ge c_0c_1^2n/2$ such that 
$$c_1/\sqrt{2n}\le |x_k| \le 1/\sqrt{c_0n}, k\in \sigma.$$
\end{claim}





We now pass our estimate to $\spread_J$.

\begin{lemma}\label{lemma:incomp:2}
Let $c_0,c_1\in (0,1)$. There exist $C,c>0$ which depend only on $c_0,c_1$ such that the following holds. Then for any $\ep>0$
$$\P\Big(\inf_{\Bx\in \Incomp(c_0,c_1)} \|B\Bx\|_2 <c \ep \sqrt{\frac{d}{n}}\Big) \le C^d \max_{J\in [n]^{d}}\P\Big(\inf_{\Bz\in \spread_J}\dist (B\Bz,H_{J^c})<\ep \Big),$$  
\end{lemma}

Where $H_{J^c}$ is the subspace generated by the columns of $B$ indexed by $J^c$.

We remark that there is a slight difference between this result and Lemma \cite[Lemma 6.2]{RV-rec} in that we take the supremum over all choices of $J$, as in this case the distance estimate for each $J$ is not identical.

Note that the following proof gives $K_1=c_1 \sqrt{c_0/2}, K_2=1/K_1, c= c_2/\sqrt{2}, C= 2e/c_0c_1^2$.

\begin{proof}(of Lemma \ref{lemma:incomp:2})
Let $\Bx \in \Incomp(c_0,c_1)$. For every $J$ we have
$$\|B\Bx\|_2 \ge \dist(B\Bx,H_{J^c}) =\dist (BP_J\Bx,H_{J^c}).$$
Condition on $E(\Bx)$ of Lemma \ref{lemma:incomp:1}, we have $\Bz=P_J\Bx/\|\P_J\Bx\|_2\in \spread_J$, and
$$\|B\Bx\|_2 \ge \|P_J \Bx\|_2 \times \inf_{\Bz\in \spread_J} \dist(B \Bz,H_{J^c})= \|P_J \Bx\|_2 D(B,J)$$ 
with
$$D(B,J):=  \inf_{\Bz\in \spread_J} \dist(B \Bz,H_{J^c}).$$
Thus on $E(\Bx)$,
$$\|B\Bx\|_2 \ge (c\sqrt{d/N}) D(B,J).$$
Define the event 
\begin{equation}\label{eqn:incomp:passing1}
\CF:=\{B: P_J(D(B,J) \ge \ep)>1-c^d\}.
\end{equation}
Markov's inequality then implies that 
\begin{align*}
\P_B(\CF^c) \le c^{-d} \E_B\P_J(D(B,J)< \ep) &\le c^{-d} \E_J\P_B(D(B,J)< \ep)\\
&\le c^{-d} \max_J P_B(D(B,J)<\ep).
\end{align*}
Fix any realization of $B$ for which $\CF$ holds, and fix any $\Bx\in \Incomp(c_0,c_1)$. Then
$$P_J(D(B,J)\ge \ep) + \P_J(E(\Bx)) \ge (1-c^d) + c^d >1.$$
Thus for any $\Bx$ there exists $J$ such that $E(\Bx)$ and $D(B,J)\ge \ep$. We then conclude from \eqref{eqn:incomp:passing1} that for any $B$ for which $\CF$ holds,
$$ \inf_{\Bx\in \Incomp(c_0,c_1)} \|B\Bx\|_2 \ge \ep c \sqrt{d/n},$$
completing the proof. 
\end{proof}

By Lemma \ref{lemma:incomp:2}, we need to study $\P(\inf_{\Bz\in \spread_J}\dist (B\Bz,H_{J^c})<\ep)$ for any fixed $J\subset [n]^d$. From now on we assume that $J=\{N-d+1,\dots,N\}$ as the other cases can be treated similarly. We restate the problem below.

\begin{theorem}\label{theorem:distance:specialJ:1}
Let $B$ be the matrix of the last $n$ columns of $A$, and let $J=\{N-d+1,\dots,N\}$. Then 
$$\P\Big(\inf_{\Bz\in \spread_J}\ \dist(B\Bz, H_{J^c}) \le \ep \Big) \le \ep^{cm} + O(\exp(-n^{\ep_0})),$$
for some absolute constanst $c,\eps_0>0$.
\end{theorem}

Notice that $H_{J^c}$ has co-dimension $N-(n-d)=m+d$ in $\R^N$, thus $\ep^{\Theta(m)}$ is expected in the RHS in Theorem \ref{theorem:distance:specialJ:1}. 

As there is still dependence between $B\Bz$ and $H_{J^c}$, we will delete the last $m$ rows from $B$ to arrive at a matrix $B'$ of size $N-d$ by $n$. That is $$B = \begin{pmatrix}\col_{m+1}(A) & \dots & \col_N(A)\end{pmatrix} = \begin{pmatrix} B' \\ \row_{N-d+1} \\ \dots \\ \row_N \end{pmatrix}$$
and
$$B' =\begin{pmatrix}a_{1(m+1)} & a_{1(m+2)} & \dots & a_{1N} \\ a_{2(m+1)} & a_{1(m+2)} & \dots & a_{2N} \\ \dots & \dots & \dots & \dots \\ a_{(N-d)(m+1)} & a_{(N-d)(m+2)} & \dots & a_{(N-d)N}  \end{pmatrix} .$$

We now ignore the contribution of distances from the last $d$ rows by an easy observation

\begin{fact}\label{claim:passing}
Assume that $B$ is as in Theorem \ref{theorem:distance:problem}, and $B'$ is obtained from $B$ by deleting its last $d$ rows. Then for any $\Bz\in \spread_J$,  
$$\dist(B\Bz, H_{J^c}) \ge \dist(B'\Bz, H_{J^c}(B')),$$
where $H_{J^c}(B')$ is the subspace spanned by the first $n-d$ columns of $B'$.
\end{fact}

By Fact \ref{claim:passing}, to prove Theorem \ref{theorem:distance:specialJ:1}, it suffices to show
$$\P\Big(\inf_{\Bz\in \spread_J}\ \dist(B'\Bz, H_{J^c}(B')) \le \ep  \Big) \le \ep^{cm} + O(\exp(-n^{\ep_0})) .$$
Note that the matrix $B''$ generated by the columns of $B'$ indexed from $J^c$ has size $(N-d)\times (n-d)$, and thus $H_{J^c}(B')$ has co-dimension $(N-d)-(n-d)=m$ in $\R^{N-d}$. Also 
\[B'': =(\begin{array}{c}
D\\
G
\end{array}), \]
where $G$ is a random symmetric matrix (inherited from $A$) of size $N-m-d=n-d$, and $D$ is a matrix of size $m\times (n-d)$,
$$B'' =\begin{pmatrix}a_{1(m+1)} & a_{1(m+2)} & \dots & a_{1(N-d)} \\ a_{2(m+1)} & a_{1(m+2)} & \dots & a_{2(N-d)} \\ \dots & \dots & \dots & \dots \\ 
\\ a_{(m+1)(m+1)} & a_{1(m+2)} & \dots & a_{(m+1)(N-d)} \\ \dots & \dots & \dots & \dots \\ a_{(N-d)(m+1)} & a_{(N-d)(m+2)} & \dots & a_{(N-d)(N-d)}  \end{pmatrix} .$$
As for any fixed $\Bz \in \spread_J$, the vector $B'\Bz$ is independent of $H_{J^c}(B')$, and that the entries are iid copies of a random variable having the same subgaussian property as our original setting.  Our next task is to prove another variant of the distance problem in contrast to Theorem \ref{thm:distance:dependent}.

\begin{theorem}[Distance problem, lower bound]\label{theorem:distance:problem}
Let $B'$ be as above.  Let $\Bx=(x_1,\dots,x_{N-d})$ be a random vector where $x_i$ are iid copies of a subgaussian random variable of mean zero and variance  one and are independent of $H_{J^c}(B'))$, then 
$$\P\Big(\dist(\Bx, H_{J^c}(B'))  \le \ep \sqrt{m} \Big) \le \ep^{cm} + O(\exp(-n^{\ep_0})),$$
for some absolute positive constants $c,\eps_0$.
\end{theorem}

We remark that the upper bound of distance is now $\eps \sqrt{m}$ as we do not normalize $\Bx$. Notice that Theorem \ref{theorem:distance:problem} is equivalent with 
\begin{equation}\label{eqn:incomp:distance:equivalent}
\P\Big(\dist(\Bx, H_{J^c}(B'))  \le t \sqrt{d} \Big) \le (t/\sqrt{\delta})^{(c \delta^{-1})d} + O(\exp(-n^{\ep_0})),
\end{equation}
where 
$$t:=\ep \sqrt{\delta^{-1}}.$$
We will prove Theorem \ref{theorem:distance:problem} in Sections \ref{section:distance:proof:preparation} and \ref{section:distance:proof:LCD}. Assuming it for now, we can pass back to  $\P(\inf_{\Bz\in \spread_J}\dist (B\Bz,H_{J^c})<\ep)$ to complete the proof of Theorem \ref{theorem:distance:specialJ:1}. First of all, for short let $P$ be the projection onto $(H_{J^c}(B'))^\perp$, and let $W$ be the random matrix $W=PB'|_{\R^J}$. Notice that for $\Bz\in \spread_J$,
$$\dist(B'\Bz,(H_{J^c}(B'))^\perp) = W \Bz.$$
By Theorem \ref{thm:distance:independent}
\begin{equation}\label{eqn:incomp:norm}
\P(\|W\|_2 \ge K \sqrt{d}) \le \exp(-CK^2 d)
\end{equation}
for any $K$ sufficiently large. 

To finish the proof, we rely on the following result from \cite{RV-rec}
where we will choose $K=C_0$ sufficiently large so that the RHS $\exp(-cK^2 d)$ of \eqref{eqn:incomp:norm} is much smaller than $C^{-m}$ from Lemma \ref{lemma:incomp:2}, for instance one can take
\begin{equation}\label{eqn:incomp:K_0}
C_0 = (c^{-1}) \log C m/d =  (c^{-1}) (\log C) \delta^{-1}.
\end{equation}

\begin{lemma}\cite[Lemma 7.4]{RV-rec}\label{lemma:incomp:3}
Assume that $W$ is the projection $W=PB|_{\R^J}$, then for an $t\ge \exp(-N/d)$ we have
$$\P\Big(\inf_{\Bz\in \spread_J}\|Wz\|_2<t\sqrt{d} \mbox{ and } \|W\|_2\le K_0\sqrt{d}\Big )\le (Ct)^{cm} + \exp(-K_0^2d).$$
\end{lemma}



\section{Proof of Theorem \ref{theorem:distance:problem}: preparation}\label{section:distance:proof:preparation}

Without loss of generality, we restate the result below by changing $n$ to $N$ and $d$ to $m$.

\begin{theorem}[Distance problem, again]\label{theorem:distance:problem'}
Let $B$ be the matrix obtained from $A$ by removing its first $m$ columns.  Let $H$ be the subspace generated by the columns of $B$, and let $\Bx=(x_1,\dots,x_N)$ be a random vector independent of $B$ whose entries are iid copies of a subgaussian random variable of mean zero and variance one, then 
$$\P(\dist(\Bx, H)  \le \ep \sqrt{m} ) \le (\delta \ep)^{m} + O(\exp(-N^{\ep_0})),$$
for some absolute constants $\delta,\eps_0 >0$.
\end{theorem}
 
After discovering Theorem \ref{theorem:distance:problem'}, the current author had found that this result is similar to \cite[Theorem 8.1]{RV-gap} (although the result of \cite{RV-gap} has a lower bound restriction on $\eps$.) As the proof here is short thanks to the simplicity of our model, we decide to sketch here for the sake of completion.
Recall that for any random variable $S$, then the Levy concentration of radius $r$ (or small ball probability of radius $r$) is defined by
$$\CL(S,r)=\sup_{s}\P(|S-s|\le r).$$
\subsection{The least common denominator} Let $\Bx=(x_1,\dots,x_N)$. Rudelson and Vershynin \cite{RV-rec} defined the 
essential {\em least common denominator} ($\LCD$) of $\Bx\in \R^N$  as follows. Fix parameters $\alpha$  and $\gamma$, where $\gamma \in (0,1)$, and define
$$
\LCD_{\alpha,\gamma}(\Bx)
:= \inf \Big\{ \theta > 0: \dist(\theta \Bx, \Z^N) < \min (\gamma \| \theta \Bx \|_2,\alpha) \Big\}.
$$
We remark that for convenience we do not require $\|\Bx\|_2$ to be larger than 1, and it follows from the definition that for any $\delta>0$,
$$\LCD_{\alpha,\gamma}(\delta \Bx) \le \delta^{-1} \LCD_{\alpha,\gamma}(\Bx).$$ 

\begin{theorem}\cite{RV-rec}\label{theorem:RV:1}
  Consider a vector $\Bx\in \R^N$ which satisfies
$\|\Bx\|_2 \ge 1$. Then, for every $\alpha > 0$ and $\gamma \in (0,1)$, and for
  $$
  \ep \ge \frac{1}{\LCD_{\alpha,\gamma}(\Bx)},
  $$
  we have
  $$
  \CL(S,\ep) \le C_0(\frac{\ep}{\gamma} + e^{-2\alpha^2}),
  $$
where $C_0$ is an absolute constant depending on the sub-gaussian parameter of $\xi$.
\end{theorem}

The definition of the essential least common denominator above
can be extended naturally to higher dimensions. To this end,  consider $d$ vectors $\Bx_1=(x_{11},\dots,x_{1N}),\dots, \Bx_m=(x_{m1},\dots, x_{mN}) \in \R^N$. Define $\By_1=(x_{11},\dots,x_{m1}),\dots, \By_n=(x_{1N},\dots,x_{mN})$ be the corresponding vectors in $\R^m$. Then we define, for $\alpha > 0$ and $\gamma \in (0,1)$,
$$ \LCD_{\alpha,\gamma}(\Bx_1,\dots, \Bx_m)$$
$$:= \inf \Big\{ \|\Theta\|_2: \; \Theta \in \R^m,
     \dist(( \langle \Theta,\By_1 \rangle, \dots, \langle \Theta,\By_N \rangle ) , \Z^N) < \min(\gamma\| (\langle \Theta,\By_1 \rangle, \dots, \langle \Theta,\By_N \rangle ) \|_2,\alpha) \Big\}.
$$
The following generalization of Theorem \ref{theorem:RV:1} gives a bound on the small ball probability for the random sum $S=\sum_{i=1}^N a_i \By_i$, where $a_i$ are iid copies of $\xi$, in terms of the additive structure of the coefficient sequence $\Bx_i$.
\begin{theorem}[Diophatine approximation, multi-dimensional case]\cite{RV-rec}\label{theorem:RV:d}
  Consider $d$ vectors $\Bx_1,\dots, \Bx_m$ in $\R^N$ which satisfies
  \begin{equation}\label{eqn:iso}
    \sum_{i=1}^N \langle \By_i, \Theta \rangle ^2 \ge \|\Theta\|_2^2
    \qquad \text{for every $\Theta \in \R^m$,}
  \end{equation}
  where $\By_i=(x_{i1},\dots, x_{im})$. Let  $\xi$ be a random variable such that $\sup_{a}\P(\xi \in B(a,1)) \le 1-b$ for some $b > 0$ and $a_1, \ldots, a_N$ be iid copies of $\xi$. Then, for every $\alpha > 0$ and $\gamma \in (0,1)$, and for
  $$
  \ep \ge \frac{\sqrt{m}}{\LCD_{\alpha,\gamma}(\Bx_1,\dots, \Bx_m)},
  $$
  we have
  $$
  \CL(S,\ep \sqrt{m}) \le \Big( \frac{C\ep}{\gamma \sqrt{b}} \Big)^m + C^m e^{-2b\alpha^2}.
  $$
\end{theorem}

We next introduce the definition of LCD of a subspace.
\begin{definition} Let $H\subset \R^N$ be a subspace. Then then LCD of $H$ is defined to be
 $$\LCD_{\alpha,\gamma}(H):=\inf_{\By_0\in H,\|\By_0\|_2=1} \LCD_{\alpha,\gamma}(\By_0).$$
\end{definition}
In what follows we prove some useful results regarding this $LCD$.
\begin{lemma}\label{lemma:subspace:LCD} Assume that $\|\Bx_1\|_2=\dots=\|\Bx_m\|_2=1$. Let $H\subset \R^N$ be the subspace generated by $\Bx_1,\dots, \Bx_m$. Then 
$$\sqrt{m} \LCD_{\alpha,\gamma}(\Bx_1,\dots,\Bx_m) \ge \LCD_{\alpha,\gamma}(H).$$
\end{lemma}

\begin{proof}(of Lemma \ref{lemma:subspace:LCD}) Assume that 
$$\dist(( \langle \Theta,\By_1 \rangle, \dots, \langle \Theta,\By_N \rangle ), \Z^N) <  \min(\gamma\|( \langle \Theta,\By_1 \rangle, \dots, \langle \Theta,\By_N \rangle ) \|_2,\alpha) .$$
Set $\By_0:= \frac{1}{t}(\theta_1 \Bx_1 +\dots + \theta_m \Bx_m)$ where $t$ is chosen so that $\|\By_0\|_2=1$. By definition
\begin{align*}
\dist(t\By_0,\Z^N) &=\dist(( \langle \Theta,\By_1 \rangle, \dots, \langle \Theta,\By_N \rangle ), \Z^N) \\
&<  \min(\gamma\|( \langle \Theta,\By_1 \rangle, \dots, \langle \Theta,\By_N \rangle ) \|_2,\alpha) \\ 
&=  \min(\gamma\|t  \By_0 \|_2,\alpha).
\end{align*}
On the other hand, as $\|\Bx_i\|_2=1$, one has 
$$\|\theta_1 \Bx_1 +\dots + \theta_m \Bx_m\|_2 \le |\theta_1|+\dots + |\theta_m| \le \sqrt{m}\|\Theta\|_2.$$
So, 
$$t \le \sqrt{m} \|\Theta\|_2.$$
Hence,
$$\LCD_{\alpha, \gamma}(\By_0)\le \sqrt{m} \LCD_{\alpha,\gamma}(\Bx_1,\dots,\Bx_m).$$
\end{proof}


\begin{corollary}\label{cor:distance:LCD}
Let $H\subset \R^N$ be a subspace of co-dimension $m$ such that $\LCD(H^\perp)\ge D$ for some $D$. Let $\Ba=(a_1,\dots,a_N)$ be a random vector where $a_i$ are iid copies of $\xi$. Then for any $\ep\ge m/D$
$$\P(\dist(\Ba,H)\le \ep \sqrt{m}) \le \Big( \frac{C\ep}{\gamma \sqrt{b}} \Big)^m + C^m e^{-2b\alpha^2}.$$
\end{corollary}

\begin{proof} Let $\Be_1,\dots,\Be_m$ be an orthogonal basis of $H^\perp$ and let $M$ be the matrix of size $m\times N$ generated by these vectors. By Lemma \ref{lemma:subspace:LCD}, 
$$\LCD_{\alpha,\gamma}(\Be_1,\dots,\Be_m) \ge D/\sqrt{m}.$$
Also, by definition
$$\dist(\Ba,H) = \|M \Ba\|_2.$$
Thus by Theorem \ref{theorem:RV:d}, for $\ep \ge \frac{m}{D}$, we have
  $$
  \CL(M\Ba,\ep \sqrt{m}) \le \Big( \frac{C\ep}{\gamma \sqrt{b}} \Big)^m + C^m e^{-2b\alpha^2}.
  $$
\end{proof}

Now we discuss another variant of arithmetic structure which will be useful for matrices of correlated entries.

\subsection{Regularized LCD}  Let $\Bx=(x_1,\dots,x_N)$ be a unit vector. Let $c_{\ast}, c_0,c_1$ be given constants. We assign a subset $\spread(\Bx)$ so that for all $k\in \spread(\Bx)$,
$$\frac{c_0}{\sqrt{N}}\le |x_{k}| \le \frac{c_1}{\sqrt{N}}.$$
Following Vershynin \cite{V} (see also \cite{NgTV}), we define another variant of LCD as follows. 
\begin{definition}[Regularized LCD]\label{def reg LCD}
  Let $\lambda \in (0, c_{\ast})$. 
  We define the {\em regularized LCD} of a vector $\Bx \in \Incomp(c_0,c_1)$ as
  $$
  \LCDhat_{\alpha,\gamma}(\Bx,\lambda) = \max \Big\{ \LCD_{\alpha,\gamma} \big(\Bx_I/\|x_I\|_2\big) : \, I \subseteq \spread(\Bx), \, |I| = \lceil \lambda N \rceil \Big\}.
  $$
We will denote by $I(\Bx)$ the maximizing set $I$ in this definition.
\end{definition}
Note that in our later application $\lambda$ can be chosen within $n^{-\lambda_0} \le \lambda \le \lambda_0$ for some sufficiently small constant $\lambda_0$. 

From the definition, it is clear that if $\LCD(\Bx)$ small then so is $\LCDhat(\Bx)$ (with slightly different parameter).
\begin{lemma}\label{lemma:comparison:regularized} For any $x\in S^{N-1}$ and any $0<\gamma < c_1\sqrt{\lambda}/2$, we have 
$$\LCDhat_{\alpha, \gamma (c_1\sqrt{\lambda}/2)^{-1}}(\Bx,\lambda) \le \frac{1}{c_0}\sqrt{\lambda}\LCD_{\alpha,\gamma}(x).$$
Consequently, for any $0<\gamma <1$
$$\LCDhat_{\kappa, \gamma}(x,\alpha) \le \frac{1}{c_0}\sqrt{\alpha}\LCD_{\kappa,\gamma (c_1\sqrt{\alpha}/2)}(x).$$
\end{lemma}
\begin{proof}(of Lemma \ref{lemma:comparison:regularized}) See \cite[Lemma 5.7]{NgTV}.
\end{proof}
We now introduce a result connecting the small ball probability with the regularized LCD.
\begin{lemma}\label{lemma:smallball:regularized} Assume that 
$$\eps\ge  \frac{1}{c_1} \sqrt{\lambda}  (\LCDhat_{\alpha,\gamma}(\Bx,\lambda))^{-1}.$$ 
Then we have 
$$\CL(S,\eps) =  O\left(\frac{ \eps}{\gamma  c_1\sqrt{\lambda}} + e^{-\Theta(\alpha^2)}\right).$$
\end{lemma}

\begin{proof} See for instance \cite[Lemma 5.8]{NgTV}.
\end{proof}

\section{Estimating additive structure and completing the proof of Theorem \ref{theorem:distance:problem'}}\label{section:distance:proof:LCD}
Again, we will be following  \cite{RV-rec,V} with modifications. A major part of this treatment can also be found in \cite[Appendix B]{NgTV} but allow us to recast here for completion.

We first show that with high probability $H^\perp$ does not contain any compressible vector, where we recall that $H$ is spanned by the column vectors of $B$. 
\begin{theorem}[Incompressible of subspace]\label{theorem:structure:incompressible}
 Consider the event $\CE_1$, 
$$\CE_1:= \{H^\perp \cap \Comp(c_0,c_1)=\emptyset \}. $$
We then have 
$$\P(\CE_1^c) \le \exp(-cn).$$
\end{theorem}
The treatment is similar to Section \ref{section:least:introduction} except the fact that we are working with $B^T$ and vectors in $\R^N$. We start with a version of Lemma \ref{lemma:compressible:1}. 
\begin{lemma}\label{lemma:compressible:1'} For every $c_0$-sparse vector $\Bx\in S^{N-1}$ one has
$$\CL(B^T\Bx, c\sqrt{N})=\sup_{\Bu}\P(\|B^T\Bx-\Bu\|_2) \le \exp(-cN).$$  
\end{lemma} 
\begin{proof} Without loss of generality, assume that the last $(1-c_0)N$ components of $\Bx$ are all zero. What remains is similar to the proof of Lemma \ref{lemma:compressible:1}.
\end{proof}
\begin{proof}(of Theorem \ref{theorem:structure:incompressible})  First of  all, there exists a $(2c_1)$-net $\CN$ of sparse vectors only of the set $\Comp(c_0,c_1)$ such that 
$$|\CN| \le (9/c_0c_1)^{c_0N}.$$  
It is not hard to show that if  there exist $\Bx\in \Comp(c_0,c_1)$ with $\|B^T\Bx-\Bu\|_2 \le c N^{1/2} \|\Bx\|_2$ and assuming  $\CE_K$, then there exists $\Bx_0\in \CN$ such that
$$\|B^T\Bx_0 - \Bv_0\|_2 \le c\sqrt{N}$$
for some $\Bv_0$.
This leaves us to estimate the probability $\P(\|B^T\Bx_0 - \Bv_0\|_2 \le c\sqrt{N})$ for each individual sparse vector $\Bx_0$, and for this it suffices to apply Lemma \ref{lemma:compressible:1'}.
\end{proof}
The main goal of this section is to verify the following result.
\begin{theorem}[Structure theorem]\label{theorem:structure}
Consider the event $\CE_2$
$$\CE_2 := \{ \forall  \By_0\in H^\perp: \LCDhat_{\alpha, c }(\By_0,\lambda) \ge N^{c/\lambda}\}.$$
We then have 
$$\P(\CE_2^c) \le \exp(-cN).$$ 
\end{theorem}
Notice that in this result, $c = \gamma (c_0\sqrt{\lambda})^{-1}$. Assume Theorem \ref{theorem:structure} for the moment, we provide a proof of our distance theorem.
\begin{proof}(of Theorem \ref{theorem:distance:problem'})
Within $\CE_2$, $\LCDhat(\By_0)$ is extremely large, and so $\LCD(H^\perp)$ is also large because of Lemma \ref{lemma:subspace:LCD} (where the factor $\sqrt{m}$ is absorbed into $N^{c/\lambda}$). We then apply Theorem \ref{theorem:RV:d} (or more precisely, Corollary \ref{cor:distance:LCD}) to complete the proof.
\end{proof}

\subsection{Proof of Theorem \ref{theorem:structure}} (See also \cite{V} and \cite[Appendix B]{NgTV}). The first step is to show that the set of vectors of small $\LCDhat$ accepts a net of considerable size. 
\begin{lemma}\label{lemma:structure}
Let $\lambda\in (c/N,c_{\ast})$. For every $D\ge 1$, the subset $\{\Bx\in \Incomp(c_0,c_1):  \LCDhat_{\alpha,c}(\Bx,\lambda) \le D \}$ has a $\alpha/D\sqrt{\lambda}$-net  $\CN$ of size 
$$|\CN|\le [CD/(\lambda N)^c]^N D^{1/\lambda}.$$
\end{lemma}
\begin{definition}
Let $D_0\ge \gamma_0\sqrt{N}$. Define $S_{D_0}$ as 
$$S_{D_0}:=\{\Bx\in \Incomp: D_0 \le \LCD_{\alpha,c}(\Bx) \le 2D_0 \},$$
where $\gamma_0$ is a constant. 
\end{definition}

\begin{lemma}\cite[Lemma 4.7]{RV-rec}
There exists a $(4\alpha/D_0)$-net of $S_{D_0}$ of cardinality at most $(C_0D_0/\sqrt{N})^N$.
\end{lemma}

One can in fact obtain a more general form as follows .
\begin{lemma}\label{lemma:moregeneral}
Let $c\in (0,1)$ and $D\ge D_0\ge c \sqrt{N}$. Then the set $S_{D_0}$ has a $(4\alpha/D)$-net of cardinality at most $(C_0D/\sqrt{N})^N$.
\end{lemma}
\begin{proof} First, by the lemma above one can cover $S_{D_0}$ by $(C_0D_0/\sqrt{N})^N$ balls of radius $4\alpha/D_0$. We then cover these balls by smaller balls of radius $4\alpha/D$, the number of such small balls is at most $(5D/D_0)^N$. Thus in total there are at most $(20C_0D/\sqrt{N})^N$ balls in total.
\end{proof}
Now we put the nets together over dyadic intervals.
\begin{lemma}\label{lemma:net:ND:1}
Let $c\in (0,1)$ and $D\ge c \sqrt{N}$. Then the set $\{X\in \Incomp(c_0,c_1): c\sqrt{N} \le \LCD_{\alpha,c}(X) \le D \}$ has a $(4\alpha/D)$-net of cardinality at most $(C_0D/\sqrt{N})^N \log_2 D$.
\end{lemma}
Notice that in the above lemmas, $\|\Bx\|_2 \ge 1$ was assumed implicitly. Using the trivial bound $\log_2 D (D/\alpha) \le D^2$, we arrive at
\begin{lemma}\label{lemma:net:ND:2}
Let $c\in (0,1)$ and $D\ge c \sqrt{N}$. Then the set $\{\Bx\in \Incomp(c_0,c_1): c\sqrt{N} \le \LCD_{\alpha,c}(Bx/\|\Bx\|_2) \le D \}$ has a $(4\alpha/D)$-net of cardinality at most $(C_0D/\sqrt{N})^N D^2$.
\end{lemma}
\begin{proof}(of Lemma \ref{lemma:structure}) Write $\Bx=\Bx_{I_0}\cup \spread(\Bx)$, where $\spread(\Bx)= I_1\cup \dots \cup I_{k_0} \cup J$ such that $|I_k| =\lambda N$ and $|J|\le \lambda N$.  
Notice that we trivially have 
$$|\spread(x)| \ge |I_1\cup \dots \cup I_{k_0}| = k_0 \lceil \alpha n \rceil \ge |\spread(x)|- \alpha n \ge c'n/2.$$ 
Thus we have 
$$ \frac{c'}{2\alpha} \le k_0 \le \frac{2c'}{\alpha}.$$
In the next step, we will construct nets for each $x_{I_j}$. For $x_{I_0}$, we construct trivially a $(1/D)$-net $\CN_0$ of size 
$$|\CN_0| \le (3D)^{|I_0|}.$$
For each $I_k$, as 
$$\LCD_{\kappa,\gamma}(x_{I_k}/\|x_{I_k}\|)\le \LCDhat_{\kappa,\gamma}(x) \le D,$$
by Lemma \ref{lemma:net:ND:2} (where the condition $\LCD_{\kappa,\gamma}(x_{I_k}/\|x_{I_k}\|)\gg \sqrt{|I_k|}$ follows the standard Littlewood-Offord estimate because the entries of $x_{I_k}/\|x_{I_k}$ are all of order $\sqrt{\alpha n}$ while $\kappa=o(\sqrt{\alpha n})$), one obtains a $(2\kappa/D)$-net $\CN_k$ of size 
$$|\CN_k|\le  \left(\frac{C_0D}{\sqrt{|I_k|}}\right)^{|I_k|} D^2.$$ 
Combining the nets together, as $x=(x_{I_0},x_{I_1},\dots,x_{I_{k_0}},x_J)$ can be approximated by $y=(y_{I_0},y_{I_1},\dots,y_{I_{k_0}},y_J)$ with $\|x_{I_j}-y_{I_j}\|\le \frac{2\kappa}{D}$, we have
$$\|x-y\|\le \sqrt{k_0+1}\frac{2\kappa}{D} \ll \frac{\kappa}{\sqrt{\alpha} D}.$$
As such, we have obtain a $\beta$-net  $\CN$, where $\beta=O(\frac{\kappa}{\sqrt{\alpha} D})$,  of size 
$$|\CN| \le 2^n |\CN_0| |\CN_1| \dots |\CN_{k_0}| \le 2^n (3D)^{|I_0|} \prod_{k=1}^{k_0} \left(\frac{CD}{\sqrt{|I_k|}}\right)^{|I_k|} D^2.$$
This can be simplified to 
$$|\CN|\le \frac{(CD)^n}{\sqrt{\alpha n}^{c'n/2}}D^{O(1/\alpha)}.$$
\end{proof}

Now we  complete the proof of Theorem \ref{theorem:structure} owing to Lemma \ref{lemma:structure} and the following bound for any fixed $\Bx$.
\begin{lemma}\cite[Proposition 6.11]{V}\label{lemma:structure:1} 
Let $\Bx\in \Incomp(c_0,c_1)$ and $\lambda\in (0,c_{\ast})$. Then for any $\ep>1/\LCDhat_{\alpha,\gamma}(\Bx)$ one has
$$\CL(B^T\Bx,\ep \sqrt{N}) \le (\frac{\ep}{\gamma \sqrt{\lambda}} +\exp(-\alpha^2))^{N-\lambda N}.$$
\end{lemma}

\begin{proof} Assume that $D\le N^{c/\gamma}$. Then with $\beta = \alpha/(D\lambda)\ge 1/D$, by a union bound
\begin{align*}
\P\Big(\exists \By_0\in H \subset S_D, \|B^T\By_0 -u\|_2 \le \beta \sqrt{N}\Big) \le \Big(\frac{\ep}{\gamma \sqrt{\lambda}} +\exp(-\alpha^2)\Big)^{N-\lambda N} \times \frac{(CD)^N}{\sqrt{\lambda n}^{c_{\ast}N/2}}D^{2/\lambda} = N^{-cN}.
\end{align*}
This completes the proof of our theorem.
\end{proof}

\section{Application: proof of Corollary \ref{cor:normal}}
 
For short, denote by $B$ the $(N-1)\times N$ matrix generated by $\row_2(A),\dots,\row_N(A)$.  We will follow the approach of \cite{RV-del, NgV-linear}.  Let $I=\{i_1,\dots,i_{m-1}\}$ be any subset of size $m-1$ of $\{2,\dots,N\}$, and let $H$ be the subspace generated by the remaining columns of $B$. Let $P_H$ be the projection from $\R^{N-1}$ onto the orthogonal complement $H^\perp$ of $H$. For now we view $P_H$ as an idempotent matrix of size $(N-1) \times (N-1)$, $P_H^2=P_H$. It is known (see for instance \cite{V,Ng-sym}) that with probability $1-\exp(-N^c)$ we have $\dim(H^\perp)=m-1$. So without loss of generality we assume $\tr(P_H)=m-1$.

Recall that by definition,
\begin{equation}\label{eqn:starting}
x_1 \col_1(B)+ x_{i_1} \col_{i_1}(B) +\dots +x_{i_{m-1}} \col_{i_{m-1}}(B) + \sum_{i\notin \{1,i_1,\dots,i_{m-1}\}} x_i \col_i (B)=0.
\end{equation}
Thus, projecting onto $H^\perp$ would then yield
$$x_1P_H \col_1(B)+ x_{i_1} P_H \col_{i_1}(B) +\dots +x_{i_{m-1}} P_H \col_{i_{m-1}}(B) = 0.$$
It follows that 
\begin{equation}\label{eqn:x_1'}
|x_1| \|P_H( \col_1(B))\|_2 = \|x_{i_1} P_H(\col_{i_1}(B)) +\dots +x_{i_{m-1}} P_H(\col_{i_{m-1}}(B))\|_2 .\end{equation}
Now if $m=C\log n$ with sufficiently large $C$, then by Theorem \ref{thm:distance:dependent} the following holds  with overwhelming probability (that is greater than $1-O(n^{-C})$ for any given $C$)
$$ \|P_H \col_{1}(B)\|_2 \asymp \sqrt{m} \mbox{ and } \|P_H \col_{i_j}(B) \|_2 \asymp \sqrt{m}, 1\le j\le m-1;$$
and hence trivially
$$  |\col_{i_{j_1}}^T(B) P_H \col_{i_{j_2}}(B)| \ll m, j_1\neq j_2.$$
Let $\CE_{I}$ be this event, on which by Cauchy-Schwarz we can bound the square of the RHS of \eqref{eqn:x_1'} by
$$ \|x_{i_1} P_H(\col_{i_1}) +\dots +x_{i_{m-1}} P_H(\col_{i_{m-1}})\|_2^2 \ll m  (\sum_{j=1}^{m-1} x_{i_j}^2) + m^{2}  (\sum_{j=1}^{m-1} x_{i_j}^2).$$
Thus we obtain
\begin{equation}\label{eqn:x_1}
|x_1| \ll m^{1/2}  (\sum_{j=1}^{m-1} x_{i_j}^2)^{1/2}.
\end{equation}

Now let $I_1,\dots, I_{\frac{n-1}{m-1}}$ be any partition of $\{2,\dots,n\}$ into subsets of size $m-1$ each (where for simplicity we assume $m-1|n-1$). Set 
$$\CE := \wedge_{1\le j\le \frac{n-1}{m-1} }\CE_{I_j}.$$
By a union bound, $\CE$ holds with overwhelming probability. Furthermore, it follows from \eqref{eqn:x_1} that on $\CE$, 
$$|x_1| \le \min \Big \{ m^{1/2} (\sum_{i\in I_j} x_i^2)^{1/2}, 1 \le j \le   \frac{n-1}{m-1}  \Big \}.$$
But as $\sum_j \sum_{i\in I_j} x_i^2 = 1-x_1^2 <1$, by the pigeon-hole principle  
$$\min \{ \sum_{i\in I_j} x_i^2, 1 \le j \le  \frac{n-1}{m-1}  \} \le \frac{m-1}{n-1}.$$ 
Thus conditioned on $\CE$,
$$|x_1| \le m^{1/2} \sqrt{\frac{m-1}{n-1}}  = O(\frac{(\log n)^{3/2}}{\sqrt{n}}).$$
The claim then follows by Bayes' identity.

\appendix

\section{Proof of Lemma \ref{lemma:manyvalues'}}\label{section:interval}

We can rely on the powerful concentration result of eigenvalues inside the bulk for random Wigner matrices from \cite{ESY}, \cite{TVuniversality} or \cite{EYY}. 

\begin{theorem}\label{theorem:concentration} Let $A$ be a random Wigner matrix as in Theorem \ref{thm:distance:independent}. Let $\eps,\delta$ be given positive constants. Then there exists a positive constant $\kappa$ such that the following holds with probability at least $1-N^{-\omega(1)}$: let $I$ be any interval of length $\log^{\kappa^{-1}}N/\sqrt{N}$ inside $[0,2-\eps]$, then  the number $N_I$ of eigenvalues $\lambda_i$ with modulus $|\lambda_i| \in I$ is well concentrated
$$|N_I - \int_{x\in I} \rho_{qc}(x)dx |\le \delta \sqrt{N} I,$$
where $\rho_{qc}(x) = \frac{2}{\pi}1_{x\in [0,2]} \sqrt{4-x^2}$ is the quarter-circle density.
\end{theorem}

As a consequence, with probability at least $1-N^{-\omega(1)}$, for any  $\log^{\kappa'}N \ll m \ll N$, any interval $[x_0+C_1m/N^{1/2},x_0+C_2m/N^{1/2}], x_0\ge 0$ inside the bulk contains at least $2m$ and at most $C' m$ singular values of $A$, where $C_1,C_2,C'$ depend on $\delta,\eps,K_0$. Lemma \ref{lemma:manyvalues'} then can be obtained by iterating the Cauchy interlacing law for eigenvalues of submatrices of Hermitian matrices, noting if $A_2$ is obtained from $A_1$ by removing one of its row, then $A_2A_2^\ast$ is a principle submatrix of $A_1A_1^\ast$.

{\bf Acknowledgments.} The author is grateful to the anonymous referees for many invaluable corrections and suggestions that help substantially improve the presentation of the note.


\begin{thebibliography}{99}
 \bibitem{GN} G.~Bennett, L.~E.~Dor, V.~Goodman, W.~B.~Johnson and C.~M.~Newman, On uncomplemented subspaces of $L_p, 1 <p < 2$, {\it Israel J. Math}. {\bf 26} (1977), 178-187.
 
 
\bibitem{ESY}
L.~Erd\H{o}s, B.~Schlein and H.-T.~Yau, Wegner estimate and level repulsion for Wigner random matrices, \emph{International Mathematics Research Notices}, 2010, no. 3, 436-479.

\bibitem{EYY} L.~Erd\H{os}, H.-T.~Yau and J.~Yin, Rigidity of Eigenvalues of Generalized Wigner Matrices, {\it Advances in Mathematics}, \textbf{229} (2012), no. 3, 1435-1515.

 \bibitem{HJ} R. A. Horn and C. R. Johnson, Matrix Analysis, Cambridge University Press, first edition, 1990.
 
\bibitem{Ng-sym} H.~ Nguyen, Inverse Littlewood-Offord problems and the singularity of random symmetric matrices, {\it Duke Mathematics Journal}, \textbf{161}, 4 (2012), 545-586.

\bibitem{NgV-linear} H.~Nguyen and V.~Vu, Normal vector of a random hyperplane, to appear in IMRN, \url{arxiv.org/abs/1604.04897}.

\bibitem{NgV-CLT} H.~Nguyen and V.~Vu, Random matrices: law of the determinant, {\it Annals of Probability}, \textbf{42} (2014), no. 1, 146-167.

\bibitem{NgTV} H.~Nguyen, V.~Vu and T~Tao, Random matrices: tail bounds for gaps between eigenvalues, to appear in Probability Theory and Related Fields, \url{arxiv.org/abs/1504.00396} .
 
 

\bibitem{RV-advances} M. Rudelson and R. Vershynin, The Littlewood-Offord Problem and invertibility of random matrices, {\it Advances in Mathematics}, \textbf{218} (2008), 600-633.

\bibitem{RV-rec} M.~Rudelson and R.~Vershynin, Smallest singular value of a random rectangular matrix, {\it Communications on Pure and Applied Mathematics},  \textbf{62} (2009), 1707-1739.

\bibitem{RV-HW} M.~Rudelson and R.~Vershynin, Hanson-Wright inequality and sub-gaussian concentration, {\it Electronic Communications in Probability}, \textbf{18} (2013), 1-9.

\bibitem{RV-del} M.~Rudelson and R.~Vershynin, Delocalization of eigenvectors of random matrices with independent entries, {\it Duke Mathematical Journal}, to appear, \url{arXiv:1306.2887}. 

\bibitem{RV-gap} M.~Rudelson and R.~Vershynin, No-gaps delocalization for general random matrices, \url{http://arxiv.org/abs/1506.04012}.


\bibitem{Tao-RMT} T.~Tao, Topics in random matrix theory, {\it Graduate Studies in Mathematics}, \textbf{132}, American Mathematical Society, Providence, RI, 2012.

\bibitem{TVcir} T.~Tao and V.~Vu and appendix by M. Krishnapur, Random matrices: universality of ESDs and the circular law, {\it Annals of Probability},  \textbf{38}, (2010), no. 5, 2023-2065.

\bibitem{TV-least} T.~Tao and V.~Vu, Random Matrices: the Distribution of the Smallest Singular Values,  {\it Geometric and Functional Analysis}, \textbf{20} (2010), no. 1, 260-297.

\bibitem{TVuniversality} T.~Tao and V.~Vu, Random matrices: universality of local eigenvalue statistics, \emph{Acta Mathematica}, \textbf{206} (2011), 127-204.

\bibitem{V} R.~Vershynin, Invertibility of symmetric random matrices, \emph{Random Structures \& Algorithms}, \textbf{44} (2014), no. 2, 135-182.

\end{thebibliography}
\end{document}